\documentclass[10pt]{amsart}
\usepackage{a4wide}
\usepackage{times}
\usepackage{amsmath,amssymb,url,stmaryrd,enumerate,bbm,enumerate}
\usepackage[french,english]{babel}
\usepackage[latin1]{inputenc}
\usepackage[all]{xy}

\usepackage[T1]{fontenc}
\usepackage{todonotes}

\usepackage{hyperref}

\usepackage{mathrsfs}

\usepackage{yhmath}

\def\R{\mathbb{R}}

\def\P{\mathbb{P}}

\def\Fp{\mathbb{F}_p}

\def\Fpb{\overline{\mathbb{F}}_p}

\def\La{\Lambda}

\def\O{\mathcal{O}} 
\def\F{\mathcal{F}}

\def\*{^\times }

\def\G{\mathcal{G}}

\def\l{\lambda}

\def\a{\alpha}

\def\ph{\varphi}

\def\drt{\rightarrow}
\def\ldrt{\longrightarrow}
\def\Q{\mathbb{Q}}

\def\Qp{\mathbb{Q}_p}
\def\Qpb{\overline{\mathbb{Q}}_p}
\def\Zp{\mathbb{Z}_p}
\def\Z{\mathbb{Z}}
\def\N{\mathbb{N}}

\def\Hom{\text{Hom}}

\def\Gal{\text{Gal}}
\def\={\! = \!}

\def\spec{\text{Spec}}
\def\spf{\text{Spf}}

\def\E{\mathcal{E}}

\def\limp{\underset{\longleftarrow}{\text{ lim }}\;}
\def\limi{\underset{\longrightarrow}{\text{ lim }}\;}

\def\iso{\xrightarrow{\;\sim\;}}

\def\GL{\mathrm{GL}}

\def\xrig{\xrightarrow}

\def\M{\mathcal{M}}

\def\X{\mathfrak{X}}
\def\GG{\Gamma}
\def\Ext{\text{Ext}}
\def\bc{\backslash}

\def\spa{\hbox{Spa}}
\def\<{<\hspace{-1mm}}
\def\>{\hspace{-1mm}>}

\def\dem{{\it Démonstration. }}

\def\Fil{\mathrm{Fil}}

\def\unp{[ {\textstyle \frac{1}{p}}]}

\def\HT{\mathrm{HT}}

\def\HN{\mathrm{HN}}
\def\htt{\mathrm{ht}}

\def\Gr{\mathrm{Gr}}
\def\BT{\mathrm{BT}}
\def\Newt{\mathrm{Newt}}
\def\VectFil{\mathrm{VectFil}}

\author{Laurent Fargues}
\address{CNRS--Institut de Mathématiques de Jussieu}
\email{laurent.fargues@imj-prg.fr}
\date{}
\keywords{$p$-divisible groups, $p$-adic Hodge theory, Harder-Narasimhan filtrations, Shimura varieties}
\subjclass[2010]{14Gxx,11Lxx}

\thanks{L'auteur a bénéficié du support  du projet ERC Advanced grant 742608 "GeoLocLang". }

\begin{document}
\title{Théorie de la réduction pour les groupes $p$-divisibles}

\maketitle
\markleft{LAURENT FARGUES}
\markright{Théorie de la réduction pour les groupes $p$-divisibles}

\selectlanguage{french}

\begin{abstract}
Partant de nos travaux sur les filtrations de Harder-Narasimhan des groupes plats finis sur un corps $p$-adique, nous développons une théorie des filtrations de Harder-Narasimhan pour les groupes $p$-divisibles. On applique cela à l'étude de la géométrie des morphismes de périodes des espaces de Rapoport-Zink et à la géométrie $p$-adique des variétés de Shimura.
On définit et étudie en particulier des domaines fondamentaux pour l'action des correspondances de Hecke.
\end{abstract}

\selectlanguage{english}

\begin{abstract}
Starting from our work on Harder-Narasimhan filtrations of finite flat group schemes over a $p$-adic field, we developp a theory of Harder-Narasimhan filtrations for $p$-divisible groups. We apply this to the study of the geometry of period morphisms for Rapoport-Zink spaces and to the $p$-adic geometry of Shimura varieties. We define and study in particular some fundamental domains for the action of Hecke correspondences.
\end{abstract}

\selectlanguage{english}

\newtheorem{theo}{Théorème}
\newtheorem*{theon}{Théorème}
\newtheorem{prop}{Proposition}
\newtheorem*{propn}{Proposition}
\newtheorem{lemm}{Lemme}
\newtheorem{coro}{Corollaire}
\newtheorem*{coron}{Corollaire}
\newtheorem{defi}{Définition}
\newtheorem*{defin}{Définition}
\newtheorem{exem}{Exemple}
\newtheorem{rema}{Remarque}
\newtheorem{conj}{Conjecture}

\tableofcontents

\section*{Introduction}

\subsection*{\'Enoncé du théorème principal}

Le résultat principal de cet article est le suivant. 

\begin{theon}[Théo. \ref{theo:vrai theo principal OK}]
Soit $\M$ l'espace de Rapoport-Zink des déformations par quasi-isogénies d'un groupe $p$-divisible simple à isogénie près sur $\Fpb$ (\cite{RZ}). Notons $\M_\eta$ sa fibre générique comme espace de Berkovich et  $$\pi_{dR}:\M_\eta\drt \F$$ l'application des périodes de Hodge-de-Rham, d'image le domaine de périodes $\F^a$. Soit $\M_\eta^{ss}$ le lieu où les points de $p$-torsion de la déformation universelle est un groupe plat fini semi-stable, un domaine analytique fermé. Alors, 
\begin{enumerate}
\item 
$
\pi_{dR} ( \M_\eta^{ss}) = \F^a.
$
En d'autres termes $Hecke.\M_\eta^{ss} = \M_\eta$,
\item le morphisme $\pi_{dR |\M_\eta^{ss}/p^\Z}$ est quasi-fini,
\item si $\M_{\eta}^{s}$ désigne le lieu stable, un ouvert de $\M_\eta$, alors 
$
\pi_{dR |\M_\eta^{s}/p^\Z}: \M_\eta^{s}/p^\Z\hookrightarrow \F^{a}.
$
\end{enumerate}
\end{theon}

Ici, l'espace {\it $\M_\eta$ est l'analogue $p$-adique des espaces hermitiens symétriques.} Ces derniers sont associés à un couple $(G,X)$, où $G$ est un groupe réductif réel et $X$ une classe de conjugaison de morphisme $h:\mathbb{S}\drt G$ satisfaisant certaines conditions. L'analogue $p$-adique de $h$ est le choix de notre groupe $p$-divisible sur $\Fpb$. La variété $\F$ est une grassmannienne et l'application de périodes $\pi_{dR}:\M_\eta\drt \F$ est l'analogue des périodes de Griffiths. Dans le cas réel, l'application de périodes est un plongement d'image un ouvert simplement connexe de la variété de drapeaux. Dans notre cas, l'ouvert $\F^a$ image de $\pi_{dR}$ n'est pas simplement connexe, {\it $\pi_{dR}$ est étale et ses fibres géométriques  sont des orbites de Hecke} en bijection avec $\GL_n (\Qp)/\GL_n(\Zp)$.

Il s'agit d'une généralisation du morphisme de périodes de Gross-Hopkins (\cite{HopkinsGross}) pour les espaces de Lubin-Tate, définie par Rapoport et Zink (\cite{RZ} chap. 5). Les {\it équations différentielles de Picard-Fuchs $p$-adiques} associées sont complètement intégrables sur $\M_\eta$ et $\pi_{dR}$ est défini par la filtration de Hodge de la déformation universelle.
% Pour une famille finie de groupes discrets $(\Gamma_i)_i$, $\coprod_i \Gamma_i \bc \M_\eta$ s'identifie au lieu supersingulier dans certaines variétés de Shimura et $\Gamma_i$ à la monodromie $p$-adique des équations de Picard-Fuchs associées.
\\

Le point (1) du théorème précédent est une généralisation du corollaire 23.15 de \cite{HopkinsGross} qui concerne le cas des espaces de Lubin-Tate. Le lieu de semi-stabilité 
est ici défini comme étant le lieu où les points de $p$-torsion de la déformation universelle est un groupe semi-stable (\cite{HNgp}, en particulier le corollaire 11).

Les points (1) et (2) mis ensemble disent que les orbites de Hecke de l'ouvert admissible associé à $\M_\eta^{ss}/p^\Z$ forment un recouvrement admissible localement fini de cet espace rigide. 

Enfin, le point (3) affirme que sur l'image par $\pi_{dR}$ de l'ouvert de stabilité, le morphisme des périodes possède une section canonique.

\subsection*{Description des différentes sections}
\subsubsection*{Section \ref{ChapitreIcFI359}}

\'Etant donnés un corps $p$-adique $K$ et un groupe $p$-divisible $H$ sur $\O_K$, on définit (théo. \ref{SFKEGI39459sguzSF20refq}) une fonction concave $\HN (H) : [0,\htt \, H]\drt [0,\dim H]$. Celle-ci est obtenue par un procédé de renormalisation à partir des polygones des $H[p^n]$ pour tout $n\geq 1$. On montrera plus tard que {\it cette fonction 
concave est en fait un polygone à coordonnées de rupture entières.} 
On montre que $\HN (H)$ est un invariant de la classe d'isogénie de $H$ (prop. \ref{KGDSGIeqerET93}).

\subsubsection*{Section \ref{XVKQSFISDghhyqrfQSIZE935mireille}}

On définit une notion de {\it groupe $p$-divisible semi-stable} (déf. \ref{defi:groupe semi-stable}), puis une notion de {\it groupe de type HN} (déf. \ref{defi:type HN}). Ce sont les groupes $H$, pour lesquels les filtrations de Harder-Narasimhan de $(H[p^n])_{n\geq 1}$ s'agencent bien entre elles pour former une filtration par des sous-groupes $p$-divisibles. On montre que la catégorie des groupes $p$-divisibles à isogénie près, isogènes à un groupe de type HN,
est une \og bonne catégorie de Harder-Narasimhan\fg{} pour la fonction pente $\frac{\dim}{\htt}$ (théo. \ref{theo:structure groupes type HN}).

\subsubsection*{Section \ref{KDSalgoGisdgETP348hhdidncbdsyg}}

Partant de $H$ sur $\O_K$, on définit une suite de quotients de $H$
$$
H=H_0\twoheadrightarrow H_1 \twoheadrightarrow H_2\twoheadrightarrow \cdots \twoheadrightarrow 
$$
qui sont soit des isogénies, soit des quotients par des sous-groupes $p$-divisibles. Il s'agit d'un {\it algorithme de descente} qui, lorsqu'il converge en temps fini, fournit un groupe $p$-divisible isogène de type HN (théo. \ref{SDKGDSHGIERTPO3268SRsdgsee}). 

\subsubsection*{Section \ref{DFDKDGisrZrogospehre4329frsf}}

On définit un polygone non renormalisé de $H$, une fonction concave sur $[0,+\infty[$ (déf. \ref{defi:poly no renormalise}). Il s'agit de la borne supérieure des polygones de la collection $(H[p^n])_{n\geq 1}$. On montre que la pente limite de ce polygone coïncide avec la première pente du polygone renormalisé (coro. \ref{QDGFKIQDFDidgvoezrf249sfsf}). Ce résultat est très utile par la suite dans la preuve du théorème \ref{theo:theo principal ou pas}.

\subsubsection*{Section \ref{etudedanslecasdevaluationdicr24925f}}

On étudie en détails le cas où le corps $p$-adique de base $K$ est de valuation discrète. Dans ce cas, {\it tout groupe $p$-divisible sur $\O_K$ est isogène à un groupe de type HN}  car l'algorithme de descente s'arrête en temps fini (théo. \ref{theo:filtration HN val disc}).
Il en résulte aussitôt que, dans ce cas, le polygone renormalisé $\HN(H)$ est un polygone à points de rupture de coordonnées entières.

Lorsque, de plus, le corps résiduel de $K$ est parfait, on interprète les filtrations de Harder-Narasimhan en termes de théorie de Hodge $p$-adique. On en déduit deux résultats clefs pour la suite. Tout d'abord, le corollaire \ref{coro:filtration determinee par HT} qui dit que la filtration à isogénie près de $H$ peut se retrouver à partir de ses périodes de Hodge-Tate
$$
\a_H: T_p (H)\ldrt \omega_{H^D}\otimes \O_{\widehat{\overline{K}}}.
$$
On obtient ensuite l'inégalité entre polygone de Harder-Narasimhan renormalisé et polygone
de Newton de la fibre spéciale de la proposition \ref{prop:inegalite HN Newton cas valuation discrete}. On peut la penser comme étant une forme de {\it semi-continuité} de la variation d'un polygone \og à la Grothendieck\fg{}.
La suite de ce texte est notamment consacrée à généraliser ce type de résultat lorsque $K$ n'est plus de valuation discrète.

\subsubsection*{Section \ref{sec:Filtration des modules de HT}}

Dans cette section, le corps $K$ est algébriquement clos par contraste avec le cas précédent, où $K$ était de valuation discrète. 
On définit et étudie de nouveaux objets d'algèbre linéaire qui remplacent les groupes $p$-divisibles: les {\it modules de Hodge-Tate}. Du point de vue de la géométrie des espaces de modules, cela revient à remplacer le fibré $\omega_{H^D}$ par l'image de $\a_H\otimes 1$ (cf. rem. 5.9 de \cite{GroupesAnalytiquesII}). Les résultats des sections précédentes sur les groupes $p$-divisibles s'adaptent à ce cadre. Le point principal est le théorème \ref{theo:essentiel} qui dit que {\it l'algorithme de descente converge dans ce cadre d'algèbre linéaire}: tout module de Hodge-Tate entier est isogène à un module de Hodge-Tate de type HN.

\subsubsection*{Section \ref{sec:Retour aux gpdiv}}

Le corps $K$ est encore algébriquement clos.
On montre, en utilisant les résultats précédents sur les modules de Hodge-Tate, que 
{\it le polygone renormalisé $\HN(H)$ est un polygone à coordonnées de rupture entières} (théo. 
\ref{theo:poly renormalise est un poly}), même si à priori $H$ n'est pas forcément isogène à un groupe de type HN contrairement au cas de valuation discrète. 
On montre, de plus, que ce polygone $\HN (H)$ est en dessous du polygone de Newton (version concave) de la fibre spéciale de $H$ (théo. \ref{theo:inegalite Newton et HN}):
$$
\HN(H)\leq \Newt (H_k)^\diamond.
$$ 
La preuve utilise la courbe (\cite{Courbe}) et une notion de polygone de Harder-Narasimhan d'une modification de fibrés admissible sur celle-ci. Ce résultat montre en particulier que {\it si $H_k$ est isocline, alors $\HN (H)$ est une droite de pente $\frac{\dim H}{\htt \, H}$.}
C'est un ingrédient important du théorème principal sur les espaces de périodes.

\subsubsection*{Section \ref{sec:application aux espaces de modules}}

On démontre le théorème principal (théo. \ref{theo:theo principal ou pas}) 
qui dit que si $H$ sur $\O_K$ a sa fibre spéciale $H_k$ simple à isogénie près, alors 
$H$ est isogène à un groupe $p$-divisible semi-stable. Cela implique le théorème principal cité précédemment.

\subsubsection*{Section \ref{sec:stratification de HN des grassmaniennes}}

On définit des {\it stratifications de Harder-Narasimhan des Grassmaniennes $p$-adiques} en utilisant les modules de Hodge-Tate de la section \ref{sec:Filtration des modules de HT}. Les strates sont des sous-diamants localement fermés dont on calcul la dimension (prop. \ref{prop:calcul dimension strates HN}). En utilisant l'inégalité entre polygones de Harder-Narasimhan et polygone de Newton de la section \ref{sec:Retour aux gpdiv}, on obtient des inclusions entre strates de Newton, telles que définies dans \cite{CaraianiScholze}, et de Harder-Narasimhan. Par exemple, {\it la strate basique est contenue dans l'ouvert semi-stable.} On démontre de plus que {\it les strates non semi-stables sont paraboliquement induites} (prop. \ref{prop:induction parabolique I}). 

Par tiré en arrière via le morphisme de périodes de Hodge-Tate, ces stratifications définissent des stratifications des variétés de Shimura $p$-adiques (sec. \ref{sec:stratifications varietes de Shimura}). On espère que ces stratifications permettent d'étudier la cohomologie des variétés de Shimura dans la lignée des travaux de Caraiani-Scholze (sec. \ref{sec:perspectives}).
\vspace{.6cm}

Depuis une première version de cet article, de nombreux travaux ont été effectués sur les filtrations de Harder-Narasimhan et la théorie de Hodge $p$-adique. On a en particulier remplacé l'utilisation des espaces de Banach-Colmez (\cite{Colmez2}) par celle de la courbe (\cite{LivreIso}) dans la preuve du théorème \ref{theo:inegalite Newton et HN}. Dans la section \ref{sec:stratification de HN des grassmaniennes} on utilise la théorie des diamants de Scholze (\cite{ScholzeCohomologyDiamonds}) afin de mettre des structures géométriques sur les strates de Harder-Narasimhan. Du point des filtrations de Harder-Narasmihan, citons les travaux de Shen (\cite{XuHN}), de Cornut et Peche Irissarry (\cite{CornutMacarena})
et de Levin et Ericksson (\cite{LevinEricksson}). Du point de vue des applications des fonctions degré des groupes plats finis aux variétés de Shimura citons les travaux de Bijakowski, Pilloni et Stroh (\cite{BPS}). Du point de vue des domaines fondamentaux pour l'action des correspondances de Hecke citons les travaux de Shen (\cite{XuCell}).
Enfin, notons les application à la géométrie diophantienne de Mocz (\cite{Lucia}).

\section{Le polygone de Harder-Narasimhan renormalisé d'un groupe
  $p$-divisible}\label{ChapitreIcFI359}
  
On reprend les notations de \cite{HNgp}. Soit $K|\Qp$ valué complet pour une valuation $v$ à valeurs dans $\R$, normalisée de manière à ce que $v(p)=1$.
{\it On ne suppose pas que la valuation soit discrète.} 
 Rappelons que si $G$ est un schéma en groupe (commutatif) fini et plat sur $\spec (\O_K)$, d'ordre une puissance de $p$ non nul, on note 
\begin{eqnarray*}
\deg (G) &=& \sum_i v(a_i) \ \text{ si }\ \omega_G\simeq \bigoplus_i \O_K /(a_i), \\
\mathrm{ht} (G ) &=& n \ \text{ si } \ |G|=p^n, \\
\mu (G) &=& \frac{\deg (G)}{\htt (G)}  \in [0,1]. 
\end{eqnarray*}
Les filtrations de Harder-Narasimhan sont prises relativement à la fonction pente $\mu$. On note 
$$
\HN (G): [0, \htt (G)]\ldrt [0,\deg (G)]
$$
le polygone de Harder-Narasimhan de $G$.

\subsection{Un résultat de convergence}

\'Etant donné un groupe $p$-divisible $H$ sur $\spec (\O_K)$, on va définir une fonction concave 
obtenue par un procédé de renormalisation à partir de la collection de polygones 
$$
 \HN (H[p^n]), \ n\geq 1.
$$

\begin{defi}[produit de convolution tropical]\label{def:produit de convolution tropical}
Soient $h,h'\geq 0$ deux nombres réels.
Soit $f$, resp. $g$, une fonction bornée sur l'intervalle $[0,h]$,
resp. $[0,h']$, à valeurs  dans $\R$. On pose
\begin{eqnarray*}
f\circledast g: [0,h+h'] &\ldrt & \R \\
x & \longmapsto & \underset{a+b=x}{\sup} f(a)+g(b)
\end{eqnarray*}
où $a,b\in \R$ satisfont $0\leq a\leq h$ et $0\leq b\leq h'$.
\end{defi}

On vérifie que:
\begin{enumerate}
\item 
la convolution tropicale de deux
fonctions concaves est concave,
\item la transformée  de Legendre, analogue tropical de la transformée de Fourier (\cite{Courbe} sec. 1.5.1), de $f\circledast g$ est la somme des transformées de Legendre de $f$ et $g$,
\item si $f$ et $g$ sont des polygones, i.e. des fonctions affines par morceaux, alors $f\circledast g$ est un polygone obtenu par \og concaténation\fg{}. En particulier, si les abscisses de rupture de $f$ et $g$ sont des entiers, ce qui est le cas pour le polygone de Harder-Narasimhan d'un groupe plat fini,   dans la définition \ref{def:produit de convolution tropical} on peut se restreindre à prendre les variables $x,a,b$ dans $\N$.
\end{enumerate}

Voici une variante de la proposition 8 de \cite{HNgp}.

\begin{prop}\label{DKGSISG92447EGsokotriplosmp}
Pour une suite exacte de schémas en groupes finis et plats sur $\spec (\O_K)$
 $$0\ldrt G'\ldrt G\ldrt G''\ldrt 0$$
 on a 
$$
\mathrm{HN} (G)\leq \mathrm{HN}(G')\circledast \mathrm{HN} (G'')
$$
avec égalité si la suite est scindée.
\end{prop}
\begin{proof}
Soit $M\subset G$ un sous-groupe plat fini. On note $M'$, resp. $M''$, l'adhérence schématique de $M_K \cap G'_K$ dans $G'$, resp. de l'image de $M_K$ dans $G''_K$ dans $G''$. Il y a une suite 
$$
0\ldrt M'\ldrt M\ldrt M''\ldrt 0
$$
qui devient exacte en fibre générique. On en déduit que 
\begin{eqnarray*}
\deg (M) &\leq & \deg (M') + \deg (M'') \\
 &\leq & \HN (G') (\htt (M')) + \HN (G'') (\htt (M'')).
\end{eqnarray*}
En passant à la borne supérieure sur tous les sous-groupes plats finis $M$ de $G$, on obtient  le résultat.
\end{proof}

Voici maintenant un résultat élémentaire qui va nous permettre de définir le \og polygone\fg{} renormalisé d'un groupe $p$-divisible.

\begin{prop}\label{FJQGpqt92571504F19751906}
Soit $h$ un nombre réel strictement positif et $\ph_n : [0,nh]\ldrt
\R_+$, $n\geq 1$, une suite de fonctions concaves bornées telles que
pour tous $n,m\geq 1$ 
$$
\ph_{n+m} \leq \ph_n\circledast \ph_m.
$$
 Alors, la suite de fonction
\begin{eqnarray*}
[0,h] & \ldrt & \R \\
x &\longmapsto & 
\frac{1}{n} \,\ph_n (nx)
\end{eqnarray*}
converge uniformément vers une fonction concave sur l'intervalle
$[0,h]$ égale à la borne inférieure de cette suite. 
\end{prop}
\dem
Commençons par montrer que pour tout $k\geq 1$, tout $x\in [0,kh]$ et
tout $n\geq 1$,
\begin{equation}\label{eq:3}
\ph_{kn} (nx) \leq n \ph_k (x).
\end{equation}
Fixons l'entier $k$. 
On procède par récurrence sur $n$, le cas $n=1$ étant
immédiat. Utilisant l'inégalité 
$$
\ph_{k(n+1)} \leq \ph_{kn}\circledast \ph_k,
$$
on obtient 
$$
\ph_{k(n+1)}  ((n+1)x)\leq \sup_{a+b= (n+1)x \atop { 0\leq a \leq nkh
    \atop 0\leq b\leq kh}} \ph_{kn} (a) + \ph_k (b).
$$
Soient donc $a$ et $b$ tels que dans la borne supérieure
précédente. Par hypothèse de récurrence
\begin{equation}\label{eq:1}
\ph_{kn} (a)\leq n \ph_{k} \left ( \frac{a}{n}\right ).
\end{equation}
De plus, en écrivant
$$
x = \frac{n}{n+1}\frac{a}{n} + \frac{1}{n+1} b
$$
et en utilisant la concavité de $\ph_k$, on obtient
\begin{equation}\label{eq:2}
\frac{n}{n+1}\, \ph_k \left (\frac{a}{n} \right ) + \frac{1}{n+1} \,\ph_k
(b)\leq \ph_k (x).
\end{equation}
Combinant les  inégalités (\ref{eq:1}) et (\ref{eq:2})  on obtient
$$
\ph_{kn}(a)+\ph_k (b) \leq (n+1) \ph_k (x)
$$
qui nous donne  l'hypothèse de récurrence au rang $n+1$.

Soit maintenant $n_0\geq 1$ un entier fixé. Pour $n\in \N$, écrivons $n=q(n)
n_0+r(n)$ la division euclidienne de $n$ par $n_0$. De l'inégalité 
$$
\ph_n\leq \ph_{q(n)n_0}\circledast \ph_{r(n)},
$$
on tire que pour $x\in [0,h]$
$$
\ph_n (nx)\leq \sup_{a+b = nx \atop { 0\leq a \leq q(n) n_0 h\atop
    0\leq b\leq r(n)h }} \ph_{q(n)n_0} (a) + \ph_{r(n)}(b).
$$
Pour des nombres $a,b$ comme dans la borne supérieure précédente, on
peut écrire
$$
q(n) n_0 x = \frac{q(n)n_0}{n} .a + \frac{r(n)}{n} .\left (
  \frac{q(n)n_0}{r(n)} b\right ).
$$
Utilisant la concavité de $\ph_{q(n)n_0}$, on obtient

\begin{eqnarray*}
\ph_{q(n)n_0} (q(n)n_0 x) &\geq & \frac{q(n)n_0}{n} \ph_{q(n)n_0} (a) +
\frac{r(n)}{n} \ph_{q(n)n_0} \left ( \frac{q(n)n_0}{r(n)} b\right ) \\
&\geq & \frac{q(n)n_0}{n} \ph_{q(n)n_0}(a).
\end{eqnarray*}
De plus, les fonctions $\ph_{r(n)}$ sont bornées par une constante $C$
indépendante de $n$. On a donc, pour $a,b$ comme précédemment,
\begin{eqnarray*}
\ph_{q(n)n_0} (a) + \ph_{r(n)}(b) &\leq & \frac{n}{q(n)n_0} \ph_{q(n)
  n_0} (q(n) n_0
x) + C \\
&\leq & \frac{n}{n_0} \ph_{n_0} (n_0 x) +C,
\end{eqnarray*}
la deuxième inégalité résultant de l'inégalité (\ref{eq:3}) prouvée au début de
cette démonstration. Au final, on obtient qu'il existe une constante
$C$ telle que pour tout $n\geq 1$
$$
\frac{\ph_n (nx)}{n} \leq \frac{ \ph_{n_0} (n_0 x)}{n_0} + \frac{C}{n}.
$$
On en déduit que 
$$
\underset{n\drt +\infty}{\mathrm{lim}\,\mathrm{sup}} \; \frac{\ph_n (nx)}{n} \leq
\frac{ \ph_{n_0} (n_0 x)}{n_0}.
$$
Cela étant vrai pour tout entier $n_0$, on a donc
$$
\underset{n\drt +\infty}{\mathrm{lim}\,\mathrm{sup}}\; \frac{\ph_n (nx)}{n} = \underset{n\drt +\infty}{\mathrm{lim}\,\mathrm{inf}} \;\frac{\ph_n (nx)}{n}
$$
et la convergence simple s'en déduit. La convergence uniforme se
vérifie en utilisant que la constante $C$ précédente ne dépend pas de
$x$. 
\qed

\subsection{Le polygone renormalisé d'un groupe $p$-divisible}
\label{sec:poly renormalise}

 Soit  maintenant $H$ un groupe $p$-divisible sur $\spec (\O_K)$.
%\begin{defi}
%On note $\dpt{\mu_H = \frac{\dim H}{\mathrm{ht}\, H }}$.
%\end{defi}
%On a donc pour tout entier $n\geq 1$
%$$
%\mu_H =\mu (H[p^n])
%$$
Notons $d$ sa dimension  et $h$ sa hauteur. Pour tout entier
$n\geq 1$, le polygone $\mathrm{HN} (H[p^n])$ a pour point extrémal
$(nh,nd)$. 

\begin{theo}\label{SFKEGI39459sguzSF20refq}
La suite de fonctions 
\begin{eqnarray*}
[0,h] & \ldrt & [0,d] \\
x& \longmapsto &\frac{1}{n} \,\mathrm{HN} (H[p^n])(nx)
\end{eqnarray*}
converge uniformément lorsque $n\drt +\infty$ vers une fonction
continue concave croissante
$$
\mathrm{HN} (H) : [0,h]\ldrt
[0,d]
$$
 égale à la borne inférieure des fonctions précédentes et
vérifiant $HN(H) (0)=0$ et $HN(h)=d$.
\end{theo}
\begin{proof}
Pour des entiers $n,m\geq 1$, 
on applique la proposition \ref{DKGSISG92447EGsokotriplosmp}
aux suites exactes 
$$
0\ldrt H[p^n]\ldrt H[p^{n+m}]\xrig{\; p^{n}\;} H[p^m]\ldrt 0.
$$
On obtient alors que la suite de fonctions $(\mathrm{HN}[p^n])_{n\geq
  1}$ satisfait aux hypothèses de la proposition
\ref{FJQGpqt92571504F19751906}. 
\end{proof}

Pour un sous-schéma en groupes fini et plat $G\subset H$, notons 
$$
n(G) = \inf \{n\geq 1\;|\; G\subset H[p^n]\}.
$$
Alors, la fonction $\mathrm{HN}(H)$ est l'enveloppe concave des points
$$
\left ( \frac{\mathrm{ht}\, G}{n(G) h}, \frac{\deg\, G}{n(G) d} \right )
$$
lorsque $G$ parcourt les sous-groupes plats finis de $H$.

\begin{rema}
Ce résultat est à comparer avec le théorème 4.1.14 de \cite{Chen1} sur les
polygones de Harder-Narasimhan limites associés aux puissances
tensorielles d'un fibré hermitien ample.
\end{rema}

\begin{rema}
On va en fait montrer (théo. \ref{theo:poly renormalise est un poly}) que la fonction concave $\HN(H)$ est un polygone à points de ruptures de coordonnées entières. Cela justifie la terminologie de \og polygone\fg{} pour cette fonction concave.
\end{rema}

On remarquera que, si $H^D$ désigne le dual de Cartier de $H$, pour $x\in [0,\mathrm{ht}\, H]$
$$
\mathrm{HN} (H^D)(x) = \mathrm{HN} (H) ( \mathrm{ht}\, H-x)-\dim H +x.
$$

\subsection{Invariance par isogénie}

\begin{prop}\label{KGDSGIeqerET93}
Pour $H_1$ et $H_2$ deux groupes $p$-divisibles isogènes sur
$\O_K$ on a
$$
\mathrm{HN} (H_1) = \mathrm{HN} (H_2).
$$ 
\end{prop}
\dem
L'idée est que le polygone du noyau d'une isogénie entre nos deux
groupes $p$-divisibles se fait \og manger \fg{} par le processus de
renormalisation. 
Soit $f:H_1\drt H_2$ une isogénie de noyau $G$. Pour tout entier $n\geq 1$, on a deux suites
exactes 
$$
0\ldrt G\ldrt p^{-n} G\ldrt H_2[p^n] \ldrt 0
$$
$$
0\ldrt H_1[p^n]\ldrt p^{-n} G\ldrt G \ldrt 0.
$$
La première suite exacte implique que 
$$
\mathrm{HN} (p^{-n}G)\leq \mathrm{HN} (G)\circledast \mathrm{HN} (H_2[p^n]).
$$
Soit $C$ une constante telle que $\forall x\in [0,\mathrm{ht}\, G]$,
$\mathrm{HN} (G)(x)\leq C$. On a donc, pour tout $x\in [0,\mathrm{ht}\, H_2]$,
\begin{eqnarray*}
\HN (p^{-n} G)(nx) &\leq &\underset{a+b=nx \atop { 0\leq a\leq
    \mathrm{ht}\, G \atop 0\leq b\leq \, n\htt \, H_2}}{\sup} \HN
(G)(a)+ \HN (H_2[p^n])(b) \\
&\leq & C+ \underset{0\leq b\leq nx}{\sup} \HN (H_2[p^n])(b) \\
&=& C+ \HN (H_2[p^n]) (nx).
\end{eqnarray*}
On en déduit que, pour tout $x\in [0,\htt\, H_2]$,
$$
\frac{1}{n} \HN (p^{-n} G)(nx)\leq \frac{C}{n} + \frac{1}{n} \HN (H_2[p^n])(nx).
$$
La partie gauche de la seconde suite exacte, l'inclusion
$H_1[p^n]\hookrightarrow p^{-n}G$, implique que pour $x\in[0,\htt\,
  H_1]$,
$$
\HN (H_1[p^n])(nx)\leq \HN (p^{-n}G) (nx).
$$
Donc si $h=\htt\, H_1=\htt\, H_2$, pour $x\in [0,h]$ on a 
$$
\frac{1}{n} \HN (H_1[p^n])(nx) \leq \frac{1}{n}\HN (H_2[p^n])(nx) + \frac{C}{n},
$$
d'où
$$
\HN (H_1)\leq \HN(H_2).
$$
Par réflexivité de la relation d'isogénie on  en déduit le résultat.
\qed

\subsection{Semi-continuité}

\begin{prop}\label{DSQDFKqsfuetD35xezmicontienutzskdfuz}
Soit $F|\Qp$ une extension valuée complète pour une valuation discrète
et $\X$ un $\spf (\O_F)$ schéma formel localement formellement de type fini. Soit $H$ un groupe $p$-divisible sur $\X$ de dimension $d$ et
de hauteur $h$ constantes. Notons $\X^{an}$ la fibre générique de $\X$ comme $F$-espace analytique de Berkovich. Alors, si $\mathcal{P}:[0,h]\ldrt [0,d]$ est une
fonction concave telle que $\mathcal{P}(0)=0$ et $\mathcal{P}(h)=d$,
$$
\{x\in |\X^{an}|\;|\; \mathrm{HN} (H_x)\geq \mathcal{P}\}
$$
est un fermé de $|\X^{an}|$. 
\end{prop}
\dem
Cet ensemble se réécrit sous la forme
$$
\bigcap_{n\geq 1} \big \{x\in \X^{an}\;|\; \mathrm{HN} (H[p^n]_x) \geq n
\mathcal{P}(n\bullet) \big \}
$$
qui est fermé d'après le théorème 3 de \cite{HNgp}.
\qed
\\

En général l'ensemble précédent n'a aucune raison d'être le fermé
sous-jacent à un domaine
analytique fermé dans $\X^{an}$ et n'est donc pas à priori associé à un ouvert
admissible de l'espace rigide $\X^{rig}$.

\subsection{Inégalité avec le polygone de Hodge}

Du théorème 5 de \cite{HNgp} on tire la proposition suivante.

\begin{prop}
Pour $H$ un groupe $p$-divisible sur $\O_K$ on a 
$$
\HN (H)\leq \mathrm{Hodge} (H)^\diamond
$$
où $\mathrm{Hodge} (H)^\diamond$ est la \og version concave\fg{} du polygone de Hodge de $H$, de pente $1$ sur $[0,\dim H]$ et $0$ sur $[\dim H, \htt\, H]$.
\end{prop}

\section{Groupes $p$-divisibles semi-stables et de type HN sur $\O_K$}
\label{XVKQSFISDghhyqrfQSIZE935mireille}

Dans cette section $K|\Qp$ est un corps valué complet pour une
valuation à valeurs dans $\R$, normalisée de telle manière que $v(p)=1$.

\subsection{Groupes $p$-divisibles semi-stables}

\begin{defi}\label{defi:groupes semi stables}
Soit $H$ un groupe $p$-divisible non nul, de hauteur et de dimension constantes sur un schéma
sur lequel $p$ est nilpotent ou bien un schéma formel $p$-adique. On pose alors
$$
\mu (H) = \frac{\dim H}{\mathrm{ht}\, H}.
$$
\end{defi}

On remarquera que l'invariant $\mu (H)$ ne dépend que de la classe
d'isogénie de $H$. 
Si $H$ est un groupe $p$-divisible sur $\O_K$,  pour tout
entier $n\geq 1$,
$$
\mu (H[p^n])=\mu (H).
$$

\begin{lemm}\label{lestabiltiozgroupsfontianemaroc}
Soit $H$ un groupe $p$-divisible non nul sur $\O_K$. Sont équivalents:
\begin{enumerate}
\item $H[p]$ est semi-stable,
\item pour tout entier $n\geq 1$, $H[p^n]$ est semi-stable,
\item pour tout sous-schéma en groupes fini et plat $G$ de $H$,
  $\mu (G)\leq \mu (H)$.
\end{enumerate}
\end{lemm}
\begin{proof}
Seul le fait que le premier point implique le second n'est pas
évident. Mais cela résulte, par récurrence sur $n$, de ce que la catégorie des groupes plats finis semi-stables de pente fixée est stable par extensions, couplé aux
suites exactes
$$
0\ldrt H[p]\ldrt H[p^{n}]\xrig{\;\times p\;} H[p^{n-1}]\ldrt 0
$$
lorsque $n\geq 1$ varie.
\end{proof}

\begin{defi}\label{defi:groupe semi-stable}
On appelle groupe $p$-divisible semi-stable tout groupe $p$-divisible
sur $\O_K$ satisfaisant aux conditions équivalentes du lemme \ref{lestabiltiozgroupsfontianemaroc}
précédent.
\end{defi}

On notera  que via la dualité de Cartier
$H$ est semi-stable si et seulement si $H^D$ l'est.
La proposition qui suit dit que la classe d'isogénie d'un groupe
$p$-divisible semi-stable vérifie une condition de semi-stabilité dans la catégorie des groupes $p$-divisibles à isogénie près. Cela justifie la terminologie de \og groupe $p$-divisible semi-stable\fg{} utilisée.

 \begin{prop}\label{DKGSDIqfiqPTidgddt483}
 Soit $H$ un groupe $p$-divisible semi-stable. Soit 
 $H'$ un groupe $p$-divisible isogène à $H$ et muni d'une filtration par des groupes $p$-divisibles 
$$
0\ldrt X\ldrt H' \ldrt Y\ldrt 0
$$
avec $X\neq 0$. 
Alors, $\mu (X)\leq \mu (H)$.
 \end{prop}
 \dem
 Il faut montrer qu'il n'existe pas de sous-groupe fini et plat $G$ de $H$, ainsi
 qu'un sous-groupe $p$-divisible $X\subset H/G$, tels que
 $\mu (X)>\mu (H)$. Mais si c'était le cas, soit, pour tout $n\geq 1$,
 $K_n\subset H$ le sous-groupe plat fini tel que $G\subset K_n$ et 
 $K_n/G=X[p^n]$. La suite exacte
 $$
 0\ldrt G\ldrt K_n \ldrt X[p^n]\ldrt 0
 $$
 implique que $\deg K_n = \deg G+n\dim X$ et $\mathrm{ht}\, K_n =\mathrm{ht}\,
 G+ n \,\mathrm{ht}\, X$. On a alors
 $$
 \underset{n\drt +\infty}{\lim} \mu (K_n) = \mu (X)
 $$
 et donc pour $n$ grand $\mu (K_n)>\mu (H )$, ce qui est impossible puisque
 $H$ est semi-stable.
 \qed

\subsection{La catégorie des groupes $p$-divisibles isogènes à un groupe
semi-stable de pente donnée à isogénie près}

La catégorie des groupes plats finis sur $\O_K$ semi-stables de pente fixées est abélienne (\cite{HNgp}). On va voir qu'il en est de même pour les groupes $p$-divisibles à isogénie près. {\it On emploie désormais la notation $$\BT_{\O_K}$$ pour la catégorie des groupes $p$-divisibles sur $\O_K$.}

\begin{defi}
Soit $\l\in [0,1]\cap \Q$. On note $\BT_{\O_K}^{ss,\l}\otimes \Q$ la
sous-catégorie pleine de $\BT_{\O_K}\otimes \Q$, les groupes $p$-divisibles à isogénie près sur $\O_K$, formée des groupes
$p$-divisibles isogènes à un groupe semi-stable de pente $\l$.
\end{defi}

Dans l'énoncé qui suit on note $\mathcal{A}b_{\O_K}$ la catégorie des faisceaux de groupes abéliens sur $\spec (\O_K)_{\mathrm{fppf}}$. On note $\mathcal{A}b_{\O_K}\otimes \Q$ pour la catégorie à isogénie près associée. On peut la voir comme le quotient de $\mathcal{A}b_{\O_K}$ par la sous-catégorie épaisse formée des faisceaux annulés par un entier non nul. 

\begin{theo}
Soit $\l\in [0,1]\cap \Q$ fixé. 
\begin{enumerate}
\item 
 La catégorie $\BT_{\O_K}^{ss,\l}\otimes
\Q$ est une sous-catégorie abélienne de $\mathcal{A}b_{\O_K}\otimes \Q$. 
\item  Une suite
$$0\ldrt X\ldrt Y\drt Z\ldrt 0$$
dans $\BT_{\O_K}^{ss,\l}\otimes
\Q$ est exacte si et seulement si elle est isomorphe à une suite
exacte de la catégorie exacte $\BT_{\O_K}$.
\end{enumerate}
\end{theo}
\dem
Soient $X$ et $Y$ deux groupes $p$-divisibles semi-stables tels que
$\mu (X)=\mu (Y)=\l$ et $f:X\drt Y$ un morphisme. Rappelons que la
catégorie des groupes plats finis sur $\O_K$ semi-stables de pente
$\l$ est une sous-catégorie abélienne de 
$\mathcal{A}b_{\O_K}$. Notons,
  pour tout $n\geq 1$,
$$
G_n=\ker f_{|X[p^n]} \subset X[p^n],
$$
un groupe plat fini semi-stable de pente $\l$. Le noyau de $f$
dans $\mathcal{A}b_{\O_K}$ est
$$\ker f =\underset{n\geq 1}{\limi} G_n.$$
 Posons, pour $n\geq 1$,
$$
a_n = \mathrm{ht}\, (G_n/G_{n-1})
$$
où l'on a posé $G_0=0$. \'Etant donné que pour $i\geq j\geq 1$,
$G_i[p^j]=G_j$, si $i\geq j\geq 1$, le morphisme
$$
G_i/G_{i-1} \xrig{\; \times p^{i-j}\;} G_j/G_{j-1}
$$
est un monomorphisme. La suite $(a_n)_{n\geq 1}$ est donc
décroissante. Soit $n_0\in \N$ tel que pour $n\geq n_0$ on ait
$a_n=a_{n_0}$. Alors, pour $i\geq j\geq n_0$,
$$
\times p^{i-j}: G_i/G_{i-1} \iso G_j/G_{j-1}
$$
est un isomorphisme. On vérifie alors que 
$$
\ker f /G_{n_0}=
\underset{n\geq n_0}{\limi} G_n/G_{n_0} \subset X/G_{n_0}
$$
est soit nul (si $a_{n_0}=0$), soit un sous-groupe $p$-divisible semi-stable
de pente $\l$ de $X/G_{n_0}$. La suite exacte associée dans
$\BT_{\O_K}$
$$
0\ldrt \ker f /G_{n_0} \ldrt X/G_{n_0} \ldrt Z\ldrt 0
$$
est telle que $Z = \mathrm{Im}\, f \subset Y$ soit un sous-groupe
$p$-divisible semi-stable de pente $\l$ de $Y$ et $\mathrm{coker}\, f=
Y/Z$. 

On a donc vérifié que le noyau et le conoyau de $f$ dans
$\mathcal{A}b_{\O_K}\otimes \Q$ 
étaient des groupes $p$-divisibles semi-stables de pente $\l$. 
\qed

\subsection{La catégorie des groupes de type HN}

Après avoir étudié la catégorie des groupes $p$-divisibles semi-stables à isogénie près, il est naturel de s'intéresser à leurs filtrations de Harder-Narasimhan.

\begin{defi}\label{defi:type HN}
Un groupe $p$-divisible $H$ sur $\O_K$ est de type HN 
s'il possède  une filtration croissante
$(H_i)_{0\leq i\leq r}$ dans $\BT_{\O_K}$ telle que $H_0=0$,
$H_r=X$, pour $1\leq i\leq r$ le groupe $p$-divisible
$ H_i/H_{i-1}$ est semi-stable et
$$
\mu (H_1/H_0) >\mu (H_2/H_1 ) >\dots >\mu (H_r/H_{r-1}).
$$
\end{defi}

Remarquons que, 
si $H$ est de 
type HN, la filtration précédente est déterminée de façon unique
puisqu'alors pour tout entier $n\geq 1$, la filtration
$(H_i[p^n])_{0\leq i\leq r}$ de $H[p^n]$ est la filtration de
Harder-Narasimhan de $H[p^n]$. Ainsi, {\it $H$ est de type HN si et
seulement si 
les filtrations de Harder-Narasimhan des $H[p^n]$, lorsque
$n$ varie, forment une filtration par des sous-groupes
$p$-divisibles.} On appellera la filtration précédente d'un groupe de
type HN la filtration de Harder-Narasimhan. 
\\

Remarquons également que, d'après la proposition \ref{KGDSGIeqerET93}, si un groupe $p$-divisible $H'$ est isogène à
un groupe de type HN $H$ tel que dans la définition \ref{defi:type HN} précédente, alors
son polygone de Harder-Narasimhan renormalisé $\mathrm{HN} (H')$ est
le polygone concave de pentes $$\mu (H_1/H_0 ) >\mu (H_2/H_1)>\dots
>\mu (H_r/H_{r-1}),$$ avec multiplicités $$\mathrm{ht} ( H_1/H_0),\dots
,\mathrm{ht} (H_r/H_{r-1}).$$
 Réciproquement, on a le résultat suivant.

\begin{prop}\label{KdfjudsfEITfhsfroodju23459sedojour}
Soit $H$ un groupe $p$-divisible sur $\O_K$. Il est de type HN si et seulement si $\mathrm{HN}(H[p])=\mathrm{HN}(H)$. 
\end{prop}
\begin{proof}
On vérifie que $H$ est de type HN si et seulement si $\forall n\geq 1, \; \mathrm{HN} (H[p^n])=\mathrm{HN} (H[p])$. 
La proposition résulte alors des inégalités valables pour tout $H$
\begin{equation*}
\mathrm{HN} (H[p])\geq \frac{1}{n} \mathrm{HN} ( H[p^n])  (n\bullet ) \geq \mathrm{HN} (H). \qedhere
\end{equation*}
\end{proof}

\begin{prop}\label{exisutnidupsemistabe35129643}
Soient $X$ et $Y$ deux groupes $p$-divisibles de type HN sur $\O_K$ et
$(X_i)_{0\leq i\leq r}$ et $(Y_j)_{0\leq j\leq s}$ leurs filtrations
de Harder-Narasimhan. Soit $f:X\drt Y$ une isogénie. Alors $r=s$ et
$f$ induit pour $1\leq i\leq r$ des isogénies 
\begin{eqnarray*}
f_{|X_i} : X_i & \ldrt &  Y_i \\
X_i/X_{i-1} & \ldrt & Y_i/Y_{i-1}. 
\end{eqnarray*}
\end{prop}
\begin{proof}
Puisque $X$ et $Y$ sont isogènes, d'après la proposition \ref{KGDSGIeqerET93}, on a l'égalité 
  $$\mathrm{HN} (X) =\mathrm{HN} (Y).$$ Celle-ci implique que
 $r=s$ et, si $1\leq i\leq r$,
 \begin{eqnarray*}
  \mu (X_i/X_{i-1}) &=&
\mu ( Y_i/Y_{i-1} ) \\
\mathrm{ht} (X_i/X_{i-1}) & = & \mathrm{ht}
(Y_i/Y_{i-1}).
\end{eqnarray*} 
On a donc que, pour tout entier $n\geq 1$, les polygones de
Harder-Narasimhan de $X[p^n]$ et celui de $Y[p^n]$ coïncident. D'après
la proposition 10 de \cite{HNgp}, cela implique que pour $0\leq i\leq
r$, $$f_{|X_i[p^n]} : X_i[p^n]\ldrt Y_i [p^n].$$
 Cela étant vrai pour
tout $n$, $f_{|X_i}: X_i\drt Y_i$. Il reste à vérifier que $f_{|X_i}$
est une isogénie. Mais si $N\in \N$ et  $g:Y\drt X$ sont tels que
$f\circ g =p^N$,  alors $g$ vérifie la même propriété
que $f$, $g_{|Y_i}: Y_i\drt X_i$. On a alors $f_{|X_i}\circ g_{|Y_i} =
p^N$.
\end{proof}

%\begin{defi}
%On note $\BT_{\O_K}^{\mathrm{HN}}\otimes \Q$ la sous-catégorie pleine de
%$\BT_{\O_K}\otimes \Q$ formée des groupes $p$-divisibles isogènes à
%un groupe de type HN.
%\end{defi}

Mettant bout à bout les résultats précédents on obtient le théorème suivant.

\begin{theo}\label{theo:structure groupes type HN}
Soit $H\in \BT_{\O_K}$ isogène à un groupe de type HN. 
\begin{enumerate}
\item $H$ possède une unique filtration de Harder-Narasimhan pour la fonction pente $\mu=\frac{\dim}{\htt}$ dans la catégorie exacte
$\BT_{\O_K}\otimes \Q$.
\item $H$ est semi-stable pour cette fonction pente si et seulement si il est isogène à un groupe $p$-divisible semi-stable.
\item Le polygone de Harder-Narasimhan renormalisé de $H$ (sec. \ref{sec:poly renormalise}) coïncide avec son polygone dans la catégorie \og de Harder-Narasimhan\fg{} $(\BT_{\O_K}\otimes \Q,\mu)$.
\end{enumerate}
\end{theo}

\section{L'algorithme de descente vers un groupe semi-stable}\label{KDSalgoGisdgETP348hhdidncbdsyg}

On garde les hypothèses précédentes. On va montrer que, {\it si la valuation de $K$ est  discrète},  {\it tout groupe $p$-divisible $H$ sur $\O_K$ est isogène à un groupe de type HN} et même produire un  \og algorithme de descente\fg{} explicite pour construire une telle isogénie.

\subsection{Première étape de l'algorithme}

Soit $H$ un groupe $p$-divisible non nul sur $\spec (\O_K)$.
 On pose,  pour un entier $k\geq 1$,
$$
G_k = \text{le premier cran de la filtration de
  Harder-Narasimhan de } H[p^k].
$$
\begin{lemm}\label{KQIZRozru3485gouroulemmedescente}
La suite $(G_k)_{k\geq 1}$ forme une suite croissante de groupes
semi-stables de même pente. De plus, pour tous $i\geq j\geq 1$, la
multiplication par $p^j$ sur $G_i$ a un noyau plat et 
$$
G_i [p^j] = G_j.
$$
\end{lemm}
\begin{proof}
Soient donc $i\geq j\geq 1$. La multiplication par $p^j$ sur $G_i$ a un noyau plat et 
de plus $G_i[p^j]$ est semi-stable avec
$$
\mu (G_i[p^j]) =\mu (G_i).
$$
\'Etant donné que $G_i$ est le premier cran de la filtration de
Harder-Narasimhan de $H[p^i]$, et que $G_j\subset H[p^i]$,
$$
\mu (G_j)\leq \mu (G_i).
$$
\'Etant donné que $G_j$ est le premier cran de la filtration de
$H[p^j]$, et que $G_i[p^j]\subset H[p^j]$,
$$
\mu (G_i[p^j])\leq \mu (G_j).
$$
De l'égalité et des deux inégalités précédentes on tire
$$
\mu (G_i) =\mu (G_i [p^j]) = \mu (G_j).
$$
De cela on déduit que $G_i[p^j]\subset G_j$ (cf. le lemme 7 de \cite{HNgp}). L'autre inclusion s'en déduit de même. 
\end{proof}

On remarquera qu'en particulier on a, pour $k\geq 1$,
$$
p\, G_{k+1}\subset G_k.
$$
%\begin{defi}\label{QSFIIzrfu129dfGFrodocaristpm}
On note maintenant  $$\mu_{\mathrm{max}} (H):=\mu_{\mathrm{max}} (H[p])$$ la plus grande pente de $\HN (H[p])$.
On a donc, pour tout $k\geq 1$, $$\mu_{max} (H)=\mu (G_k)=\mu_{max} (H[p^k])$$ et
$$
\mu_{max} (H) = \sup\{\mu (G)\ |\  G\subset H\} = \sup \{\mu (G)\ |\ 
G\subset H[p] \}.
$$
On a toujours 
$$
\mu_{max} (H)\geq \mu (H)
$$
avec égalité si et seulement si $H$ est semi-stable.
\\

\begin{defi}\label{DGKDGIETR982583edg2}
 On pose
$$
\F_H =\underset{k\geq 1}{\limi} G_k \subset H
$$
comme sous-faisceau fppf de $H$. 
\end{defi}

Il résulte du lemme précédent que, pour
tout $k \geq 1$, $\F_H[p^k]=G_k $ est un groupe plat fini sur $\spec (\O_K)$. On a $\F_H=H$ si et seulement si $H$ est
semi-stable. 

\begin{lemm}\label{kourutopresakitekbouibui87814}
Supposons $H$ non semi-stable.
Deux possibilités se
présentent:
\begin{enumerate}
\item soit $\F_H$ est un groupe plat fini sur $\O_K$,
  c'est à dire il existe $k_0\geq 1$ tel que $\F_H =\F_H[p^{k_0}]$, 
\item soit il existe un entier $k_0\geq 0$, tel  $\F_H/\F_H[p^{k_0}]$ soit un
  sous-groupe $p$-divisible non nul de $H/\F_H[p^{k_0}]$,  
  semi-stable et vérifiant $$\mu (\F_H/\F_H[p^{k_0}] )=\mu_{\mathrm{max}} (H) >\mu (H).$$
\end{enumerate}
\end{lemm}
\begin{proof}
Posons, pour un entier $k\geq 1$,
$$
a_k = \mathrm{ht} \left ( G_{k}/G_{k-1}\right ).
$$
Si $i\geq j\geq 1$, il y a un monomorphisme
$$
 \times p^{i-j}:G_{i}/G_{i-1} \hookrightarrow G_{j}/G_{j-1}.
$$
La suite $(a_k)_k$ est donc  décroissante. Il existe donc $k_0\geq
1$ tel que pour $k\geq k_0$ on ait $a_k=a_{k_0}$. 
\begin{itemize}
\item Si $a_{k_0}=0$ alors,
 pour $k\geq k_0$, $G_{k}= G_{k_0}$. 
\item 
Si $a_{k_0}>0$, pour $i> j\geq k_0$,
$$
p^{i-j} : G_{i}/G_{i-1} \iso G_{j}/G_{j-1}
$$
est un isomorphisme. 
\end{itemize}
Posons, pour $i\geq 0$, $K_i = G_{k_0+i}/G_{k_0} \subset
\F_H/G_{k_0}$. La suite $(K_i)_{i\geq 1}$ est une suite croissante de
sous-groupes finis et plats de $\F_H/G_{k_0}$, semi-stables de pente
$\mu_{\mathrm{max}} (H)$,  telle que 
\begin{eqnarray*}
\F_H/G_{k_0} &=& \underset{i\geq 1}{\limi} K_i \\
K_i &\subset & (\F_H/G_{k_0})[p^i] \\
\forall i\geq j\geq 1,\;\;\; p^{i-j} K_i &\subset & K _j \\
p^{i-j}: K_i/K_{i-1} &\iso &  K_j/K_{j-1}.
\end{eqnarray*} 
On en déduit que $\F_H/G_{k_0}$ est un groupe $p$-divisible tel que, pour
$i\geq 1$, $(\F_H/G_{k_0})[p^i] = K_i$. 
\end{proof}

\subsection{Deuxième étape}
\label{sec:deuxieme etape}
 
Soit $H$ un groupe $p$-divisible non nul  sur $\O_K$. On
définit, par récurrence, une suite de groupes $p$-divisibles
$(H_i)_{i\geq 1}$ munie de morphismes 
$$
\ph_i : H_i \ldrt H_{i+1}.
$$
On pose pour cela $H_1=H$ et, si $H_i\neq 0$, $$H_{i+1} = H_i/\F_{H_i},$$ avec pour
morphisme $\ph_i$  la projection. Si $H_i=0$ on pose $H_{i+1}=0$. 
 On a donc, pour $i\geq 1$:
\begin{itemize}
\item soit $\ph_i$ est une isogénie de noyau un groupe semi-stable
  vérifiant
\begin{eqnarray*}
\mu (\ker \ph_i) &=& \mu_{\mathrm{max}} (H_i) \\
\mu_{\mathrm{max}} (H_{i+1}) &< &\mu_{\mathrm{max}} (H_i),
\end{eqnarray*}
\item soit $\ph_i$ est la composée d'une isogénie $\psi: H_i\drt
  H'_i$, telle que $\ker \psi$ soit semi-stable de pente $\mu (\ker
  \psi)=\mu_{\mathrm{max}} (H)$, avec un morphisme $H'_i \drt H_{i+1}$ qui
  fait de $H_{i+1}$ un quotient de $H'_i$ dans la catégorie exacte
  $\BT_{\O_K}$ et tel que 
$$
\mu (H_{i+1})<\mu (H_i).
$$
\end{itemize}

\subsection{Lorsque l'algorithme s'arrête en temps fini}

Lorsque la valuation de $K$ est discrète, l'algorithme de descente précédent s'arrête en temps fini.

\begin{theo}\label{SDKGDSHGIERTPO3268SRsdgsee}
Soit $H$ un groupe $p$-divisible sur $\O_K$. 
\begin{enumerate}
\item 
Supposons que dans la suite de groupes $p$-divisibles construite
précédemment, $(H_i)_{i\geq 1}$, on ait $H_i=0$ pour $i\gg 0$. Alors, $H$
est isogène à un groupe de type $HN$.
\item Lorsque la valuation de $K$ est discrète, l'algorithme s'arrête en temps fini et tout groupe $p$-divisible sur $\O_K$ est isogène à un groupe de type HN.
\end{enumerate}
\end{theo}
\begin{proof}
Seul le point (2) demande vérification. C'est une conséquence de ce que, avec les notations de 
la section \ref{sec:deuxieme etape}, on a $\mu_{\mathrm{max}} (H_{i+1})<\mu_{\mathrm{max}} (H_i)$ qui varie dans un ensemble fini puisque $\mu_{\mathrm{max}} (H)=\mu_{\mathrm{max}} (H[p])$ pour $H\in \BT_{\O_K}$.
\end{proof}

\section{Polygone de Harder-Narasimhan non renormalisé}
\label{DFDKDGisrZrogospehre4329frsf}

Soit $H$ un groupe 
$p$-divisible non nul sur $\O_K$. Si $n_0\geq 1$ est un entier, la suite de polygones 
$$\big ( \mathrm{HN}(H[p^n])_{| \, [0,n_0 \mathrm{ht}\, H]}\big )_{n\geq n_0}$$ est croissante. On peut donc
poser la définition suivante.

\begin{defi}\label{defi:poly no renormalise}
On appelle polygone de Harder-Narasimhan non renormalisé de $H$ la fonction concave
\begin{eqnarray*}
\mathrm{HN}(H)^{nr}: [0,+\infty [&\ldrt & [0,+\infty ] \\
x & \longmapsto & \underset{n\drt +\infty \atop n \mathrm{ht}\, H\geq x}{\lim} \mathrm{HN}(H[p^n])(x).
\end{eqnarray*}
\end{defi}

Analysons cette fonction en termes de l'algorithme de la section \ref{KDSalgoGisdgETP348hhdidncbdsyg}. 
Soit $$
H=H_1\twoheadrightarrow\cdots\twoheadrightarrow H_i\twoheadrightarrow H_{i+1}\twoheadrightarrow \cdots
$$
 la suite de 
groupes $p$-divisibles construite dans la section \ref{KDSalgoGisdgETP348hhdidncbdsyg}, où si $H_i\neq 0$ alors
$H_{i+1}=H_i/\F_{H_i}$ (cf. déf. \ref{DGKDGIETR982583edg2}). Soit $i_0\in [1,+\infty]$ tel que 
\begin{itemize}
\item $i_0=+\infty$ si $\forall i\geq 1$, $\F_{H_i}$ est un schéma en groupes fini et plat, i.e. pour tout $i\geq 1$ le morphisme $H_i\twoheadrightarrow H_{i+1}$ est une isogénie,
\item sinon, $i_0=\inf \{ i\geq 1\;|\; \F_{H_i}\text{ n'est pas un schéma en groupes fini et plat}\}$.
\end{itemize}

Prenons la définition suivante:
\begin{enumerate}
\item 
Si $i_0=+\infty$, considérons la suite strictement décroissante $(\mu_i)_{i\geq 1} = (\mu_{\mathrm{max}} (H_i))_{i\geq 1}$. On pose alors
$$
\mu_\infty (H) = \underset{i\drt +\infty}{\lim} \mu_i.
$$
\item 
Si $i_0\neq +\infty$, considérons la suite décroissante $(\mu_i)_{1\leq i\leq i_0} = (\mu_{max} (H_i))_{1\leq i\leq i_0}$. On pose alors 
$$
\mu_\infty (H)=\mu_{i_0}.
$$
\end{enumerate}

On dispose alors du lemme suivant qui permet de relier le polygone non-renormalisé de $H$ à l'algorithme de descente.

\begin{lemm}
Soit $\mathcal{P}:[0,+\infty[\ldrt [0,+\infty[$ le polygone concave défini par 
\begin{enumerate}
\item si $i_0=+\infty$, les pentes de $\mathcal{P}$ sont les $(\mu_i)_{i\geq 1}$ avec multiplicités $\big ( \mathrm{ht}\, \ker (H_i\twoheadrightarrow H_{i+1})\big )_{i\geq 1}$
\item si $i_0\neq +\infty$, les pentes de $\mathcal{P}$ sont les $(\mu_i)_{1\leq i\leq i_0}$ avec multiplicités 
$$\big ( \mathrm{ht}\, \ker (H_1\twoheadrightarrow H_2 ),\dots, \mathrm{ht}\, \ker (H_{i_0-1}\twoheadrightarrow H_{i_0}),\infty \big )$$.
\end{enumerate}
Alors, pour tout entier $n\geq 1$,
$$
\mathrm{HN} (H[p^n])_{|[0,n]} = \mathcal{P}_{|[0,n]}.
$$
\end{lemm}

On en déduit aisément le résultat qui suit.

\begin{coro}\label{QDGFKIQDFDidgvoezrf249sfsf}
Les propriétés suivantes sont vérifiées:
\begin{enumerate}
\item 
Le polygone $\mathcal{P}$ défini dans le lemme précédent coïncide avec $\mathrm{HN}(H)^{nr}$.
\item 
Sur le segment $[0,1]$ la fonction de Harder-Narasimhan renormalisée $\HN (H)$ est une droite de pente $\mu_\infty (H)=\underset{x\drt +\infty}{\lim} \partial_t \mathrm{HN} {( H)^{nr}}_{|t=x}$.
\end{enumerate}
\end{coro}
%
%
%
%\begin{coro}\label{QSDDVKDSviefg395fusfiudgf}
%Soit $F|\Qp$ valué complet pour une valuation discrète. 
%Soit $\X$ un $\spf (\O_F)$-schéma formel localement formellement de type fini et $H$ un groupe $p$-divisible sur $\X$. Alors la fonction
%\begin{eqnarray*}
%|\X^{an}| & \ldrt & \{\text{polygones }:[0,+\infty[\ldrt [0,+\infty[\} \\
%x & \longmapsto & \mathrm{HN} (H_x)^{nr}
%\end{eqnarray*}
%est continue lorsqu'on munit l'espace des polygones définis sur $[0,+\infty[$ de la topologie de la convergence uniforme sur les compacts de $[0,+\infty[$. 
%\end{coro}

\section{\'Etude dans le cas de valuation discrète}\label{etudedanslecasdevaluationdicr24925f}

Dans toute cette section {\it $K|\Qp$ est supposé de valuation discrète.} On note $k$ le corps résiduel de $K$.

\subsection{Un résultat de structure général}

En cumulant les théorèmes \ref{theo:structure groupes type HN} et \ref{SDKGDSHGIERTPO3268SRsdgsee} on obtient le résultat suivant.

\begin{theo}\label{theo:filtration HN val disc}
Considérons la catégorie exacte $\BT_{\O_K}\otimes \Q$ munie de la fonction pente $\mu= \frac{\dim}{\htt}$. Soit $H$ un groupe $p$-divisible sur $\O_K$.
\begin{enumerate}
\item $H$ possède une unique filtration de Harder-Narasimhan dans $(\BT_{\O_K}\otimes \Q,\mu)$.
\item $H$ est semi-stable dans $(\BT_{\O_K}\otimes \Q,\mu)$ si et seulement si il est isogène à un groupe semi-stable au sens de la définition \ref{defi:groupes semi stables}. 
\item Le polygone de Harder-Narasimhan de $H$ dans $(\BT_{\O_K}\otimes \Q,\mu)$ coïncide avec 
le polygone renormalisé $\HN(H)$.
\end{enumerate}
\end{theo}

Prenons maintenant la définition suivante.

\begin{defi}
On note $\Newt (H_k)^\diamond$ la \og version  concave \fg{} du polygone de Newton de la fibre spéciale de $H\in \BT_{\O_K}$. Plus précisément,
$$
\forall x\in [0,\htt (H)],\ \ \Newt(H_k)^\diamond (x) = \dim (H) - \Newt (H_k) (\htt (H)-x).
$$
\end{defi}

La proposition suivante est une remarque clef dont la démonstration est immédiate. Le point (1) entraîne immédiatement le point (2). Le point (2) est lui-même un {\it analogue du théorème de Grothendieck sur la  semi-continuité du polygone de Newton} dans les familles de groupes $p$-divisibles 
paramétrées par une base de caractéristique $p$ (\cite{KatzSlope}).

\begin{prop}\label{prop:inegalite HN Newton cas valuation discrete} 
Les propriétés suivantes sont vérifiées.
\begin{enumerate}
\item 
La filtration de Harder-Narasimhan de $H_{k}$ dans la catégorie abélienne $\BT_k\otimes \Q$, munie de la fonction pente $\mu= \frac{\dim}{\htt}$, est donnée par la filtration par les pentes de Grothendieck (\cite{Zink1}) de $H_k$.
\item On a l'inégalité
$$
\HN (H)\leq \Newt (H_k)^\diamond.
$$
\end{enumerate} 
\end{prop}

Notons le corollaire important suivant.

\begin{coro}\label{coro:isocline implique semi stable}
Si la fibre spéciale $H_k$ de $H\in \BT_{\O_K}$ est isocline, alors $H$ est isogène à un groupe $p$-divisible semi-stable. 
\end{coro}

\begin{rema}
Si on remplace la fonction pente $\mu=\frac{\dim}{\htt}$ par $-\mu$
dans la catégorie abélienne
$\BT_k\otimes\Q$,
 on obtient
également une filtration de Harder-Narasimhan.  Si $k$ est parfait, il s'agit de la
filtration opposée à la filtration précédente relativement à la
décomposition de Dieudonné-Manin en somme directe d'objets
isoclines. Mais sur un corps $k$ quelconque, la filtration par les
pentes n'est pas scindée et cette autre filtration ne semble pas avoir
d'interprétation intéressante.
\end{rema}

\subsection{Filtration de Harder-Narasimhan des représentations de Hodge-Tate et cristallines}

Dans cette sous-section on suppose de plus que {\it le corps résiduel $k$ de $K$ est parfait}. On va faire le lien entre les filtrations précédentes et certaines filtrations d'objets de théorie de Hodge $p$-adique.

\subsubsection{Filtration de Harder-Narasimhan des espaces vectoriels filtrés}

Rappelons d'abord la construction suivante. Soit $L|E$ une extension de corps. Notons $\VectFil_{L/E}$ la catégorie exacte formée des couples $(V,\Fil^\bullet V_L)$ où
\begin{itemize}
\item $V$ est un $E$-espace vectoriel de dimension finie
\item $\Fil^\bullet V_L$ est une filtration décroissante de $V_L :=
  V\otimes_E L$, telle que  $\Fil^i V_L =V_L$ pour $i\ll 0$ et $\Fil^i
  V_L=0$ pour $i\gg 0$.
\end{itemize}
Pour un tel objet on pose 
\begin{eqnarray*}
\mathrm{rg} ( V,\Fil^\bullet V_L) &:= & \dim V \\
\deg ( V,\Fil^\bullet V_L) &:= &  \sum_{i\in \Z} i. \dim \mathrm{Gr}^i V_L.
\end{eqnarray*}
On dispose alors de filtrations de Harder-Narasimhan dans $\VectFil_{L/E}$ pour la fonction pente
$$
\mu  = \frac{\deg}{\mathrm{rg}}.
$$
On renvoie pour cela à \cite{Andre1} et la section 5.5 de \cite{Courbe}. 

\subsubsection{Filtration des représentations de Hodge-Tate}
\label{sec:filtration de HT}

Notons $G_K=\Gal (\overline{K} |K)$ le groupe de Galois d'une clôture algébrique de $K$ et 
$$
\mathrm{Rep}_{\Qp}^{\mathrm{HT}} (G_K)
$$
la catégorie des représentations de $G_K$ à valeurs dans un $\Qp$-espace vectoriel de dimension fini qui sont de Hodge-Tate. On note $C=\widehat{\overline{K}}$. Pour $V\in \mathrm{Rep}_{\Qp}^{\mathrm{HT}} (G_K)$ on pose
\begin{eqnarray*}
\mathrm{rg} (V) & := & \dim_{\Qp} V \\
\deg (V) & := & d\ \text{ si } \det ( V_C) \simeq C(d).
\end{eqnarray*}
On s'intéresse aux filtrations de Harder-Narasimhan dans la catégorie abélienne $\mathrm{Rep}_{\Qp}^{\mathrm{HT}} (G_K)$ munie de la fonction pente $\mu =\frac{\deg}{\mathrm{rg}}$.

Pour $V\in \mathrm{Rep}_{\Qp}^{\mathrm{HT}} (G_K)$, définissons la filtration $\Fil^\bullet V_C$ par la formule 
$$
\Fil^i V_C = \bigoplus_{j\geq i} V_C(-j)^{G_K} \otimes_K C(j).
$$
Cela définit un foncteur exact entre catégories de Harder-Narasimhan
\begin{equation}\label{eq:foncteur HT}
\mathrm{Rep}^{\mathrm{HT}}_{\Qp} (G_K) \ldrt \VectFil_{C/\Qp}
\end{equation}
via lequel les fonctions pentes se correspondent. 

\begin{prop}
Via le foncteur (\ref{eq:foncteur HT}), la filtration de Harder-Narasimhan d'une représentation de Hodge-Tate $V$ correspond à la filtration de Harder-Narasimhan de l'espace vectoriel filtré dans $\VectFil_{C/\Q}$ associé.
\end{prop}
\begin{proof}
C'est une conséquence de ce que la fonction pente d'un objet de $VectFil_{C/\Q}$ est invariante par $\Gal (\overline{K}|K)$. Il en résulte que la filtration de Harder-Narasimhan de l'objet de $VectFil_{C/\Q}$ associé à une représentation de Hodge-Tate est automatiquement Galois invariante.
\end{proof}

Ce qui nous intéresse pour plus tard est le corollaire suivant.

\begin{coro}\label{coro:filtration determinee par HT}
La filtration de Harder-Narasimhande de $H\in \BT_{\O_K}$ dans $\BT_{\O_K}\otimes \Q$ (théo. \ref{theo:filtration HN val disc}) est complètement déterminée par son application de périodes de Hodge-Tate
$$
\a_{H}: T_p(H) \ldrt \omega_{H^D}\otimes_{\O_K} \O_C
$$
via les filtration de Harder-Narasimhan des objets de $\VectFil_{C/\Qp}$.
\end{coro}

\subsubsection{La bifiltration des $\ph$-modules filtrés}

Nous n'utiliserons par les résultats de cette sous-section dans la suite. 
Notons $K_0=W(k)_\Q$ et 
$$
\ph\mathrm{-ModFil}_{K/K_0}
$$
la catégorie des $\ph$-module filtrés associée de Fontaine. Les $\ph$-modules filtrés admissibles, qui correspondent aux représentations cristallines, sont les objets semi-stables pour la fonction degré donnée par $t_H-t_N$ où $t_H$, resp. $t_N$, est le point terminal du polygone de Hodge, resp. Newton (cf. par exemple \cite{Courbe} sec. 10.5.1). On va voir que l'on peut coupler les filtrations des sections précédentes avec ces filtrations et définir des bifiltrations des $\ph$-modules filtrés. Cela est essentiellement dû au fait que l'on ne dispose pas seulement de deux fonctions additives sur $\ph\mathrm{-ModFil}_{K/K_0}$, mais de trois.

Munissons $\Z^2$ de l'ordre lexicographique et posons
$$
\deg= (t_H-t_N, -t_N): \ph\mathrm{-ModFil}_{K/K_0} \ldrt \Z^2.
$$
La fonctions rang d'un $\ph$-module filtré est définie comme étant la hauteur de l'isocristal associé. On vérifie alors aisément que l'on dispose de filtrations de Harder-Narasimhan pour la fonction pente $\frac{\deg}{\mathrm{rg}}$ à valeurs dans $\Q^2$ muni de l'ordre lexicographique. On note $V_{cris}$ le foncteur covariant de Fontaine défini par
$$
V_{cris} ( D,\ph^,\Fil^\bullet D_K) = \Fil^0( D\otimes_{K_0} B_{cris} )^{\ph=\mathrm{Id}}.
$$
On vérifie alors aisément le résultat suivant.

\begin{prop}
Via le foncteur $V_{cris}$ covariant de Fontaine, la filtration de Harder-Narasimhan d'un $\ph$-module filtré admissible correspond à sa filtration de Harder-Narasimhan comme représentation de Hodge-Tate (sec. \ref{sec:filtration de HT}).
\end{prop}

Notons le résultat suivant qui généralise le point (2) de la proposition \ref{prop:inegalite HN Newton cas valuation discrete}, ainsi que le corollaire \ref{coro:isocline implique semi stable}, et qui est le cas où les poids de Hodge-Tate sont dans $\{0,1\}$. La terminologie \og semi-stable\fg{} entre ici en conflit avec la terminologie de Fontaine, on l'entend ici au sens des filtrations de Harder-Narasimhan.

\begin{prop}\label{prop:inegalite HN Newton cas discret}
Pour une représentation cristalline $V$ de $G_K$, d'isocristal associé $(D,\ph)$, on a l'inégalité
$$
\HN (V) \leq \Newt (D,\ph)^{\diamond}.
$$
En particulier, si $(D,\ph)$ est isocline, alors $V$ est semi-stable dans $\mathrm{Rep}_{\Qp}^{\mathrm{HT}} (G_K)$.
\end{prop}

\section{Filtration des modules de Hodge-Tate}
\label{sec:Filtration des modules de HT}

Dans cette section {\it le corps de base $K$ sera supposé algébriquement clos et on le note $C$.} On se place donc dans une situation \og non-arithmétique\fg{} purement géométrique, à l'opposé du cas de valuation discrète étudié précédemment. On va montrer que l'on peut  retrouver le polygone de Harder-Narasimhan renormalisé d'un groupe $p$-divisible sur $\O_C$, à partir de ses périodes de Hodge-Tate par une formule explicite.

\subsection{Quelques définitions}\label{sec:quelques definitions}

\begin{defi}
\begin{enumerate}
\item 
On appelle module de Hodge-Tate de torsion un triplet $(M,\omega,\a)$ où $M$ est un groupe abélien fini d'ordre une puissance de $p$, $\omega$ un $\O_C$-module de présentation finie de torsion et $\a:M\drt \omega$ vérifie 
$$
\O_C .\a (M)=\omega.
$$
 On note $\M_{\mathrm{tor}}^{\mathrm{HT}}$ la catégorie correspondante.
\item On appelle module de Hodge-Tate entier un triplet $(T,\omega,\a)$ où $T$ est un $\Zp$-module libre de rang fini, $\omega$ un $\O_C$-module libre de rang fini et $\a:T\drt \omega$
vérifie 
$$
\O_C.\a (T)=\omega.
$$
On note $\M_{\Zp}^{\mathrm{HT}}$ la catégorie correspondante.
\item On appelle module de Hodge-Tate rationnel un triplet $(V,\omega,\a)$ où $V$ est un $\Qp$-espace vectoriel de dimension finie, $\omega$ est un $C$-espace vectoriel de dimension finie et $\a:V\drt \omega$ vérifie
$$
C.\a (V)=\omega.
$$
 On note $\M_{\Qp}^{\mathrm{HT}}$ la catégorie correspondante.
\end{enumerate}
\end{defi}

Ce sont des catégories exactes. Par exemple, les suites exactes dans $\M^{\HT}_{\mathrm{tor}}$ sont données par des diagrammes 
$$
\xymatrix@C=8mm@R=8mm{
0\ar[r] & M'\ar[d] \ar[r] & M \ar[d] \ar[r] & M'' \ar[d] \ar[r] & 0 \\
0\ar[r] & \omega'\ar[r] & \omega \ar[r] & \omega''\ar[r] & 0 
}
$$
où les deux lignes horizontales sont exactes. 
L'application $(T,\omega,\a)\mapsto (T\unp,\omega\unp, \a\unp)$ induit une identification 
$$
\M_{\Zp}^{\mathrm{HT}}\otimes \Q = \M_{\Qp}^{\mathrm{HT}}.
$$
En d'autres termes, $\M_{\Qp}^{\HT}$ est la \og catégorie à isogénie près\fg{} associée à $\M_{\Zp}^{\HT}$.
Si $N$ est un $\O_C$-module de présentation finie de torsion, on note 
$$
\deg (N) := v ( \mathrm{Fitt}_0 N),
$$
la valuation d'un générateur du $0$-ième idéal de Fitting associé à $N$.
 Plus concrètement, si $N\simeq \oplus_i \O_C/a_i\O_C$,  $\deg (N)=\sum_i v(a_i)$. 

\begin{defi}\label{defi:module de Hodge Tate}
On définit des fonctions pentes de la façon suivante: 
\begin{enumerate}
\item Pour $X=(M,\omega,\a)\in \M^{\mathrm{HT}}_{\mathrm{tor}}$ on note 
\begin{eqnarray*}
\deg (X) &= & \log_p |M|-\deg (\omega) \\
\htt (X) &=& \log_p |M| \\
\mu (X ) &=& \frac{\deg(X)}{\htt (X)}\in [0,1]. 
\end{eqnarray*}
\item Pour $X=(V,\omega,\a)\in \M^{\mathrm{HT}}_{\Qp}$ on note 
\begin{eqnarray*}
\dim (X) &= & \dim V  -\dim \omega \\
\mathrm{ht} (X) &=& \dim V \\
\mu (X ) &=& \frac{\dim(X)}{\mathrm{ht} (X)}\in [0,1]\cap \Q. 
\end{eqnarray*}
\end{enumerate}
\end{defi}

À $(V,\omega,\a)\in \M^{\mathrm{HT}}_{\Qp}$ est associé l'objet $(V,\Fil^\bullet V_C)\in \VectFil_{C/\Qp}$, où l'on pose $\Fil^{0} V_C= V_C$, $\Fil^1 V_C=\ker \a\otimes 1$ et $\Fil^2 V_C=0$. Via ces formules, $\M^{\mathrm{HT}}_{\Qp}$ se plonge dans $\VectFil_{C/\Qp}$ et les fonctions pentes se correspondent.
\\

On dispose d'un foncteur \og fibre générique\fg{} (\cite{Courbe} sec. 5.1.5)
$$
(M,\omega,\a)\mapsto M
$$
de $\M^{\mathrm{HT}}_{\mathrm{tor}}$ vers la catégorie abélienne des groupes abéliens finis. Dès lors, on dispose de filtrations de Harder-Narasimhan dans cette catégorie.
Par analogie avec le cas des groupes plat finis, pour $(M,\omega,\a)\in \M^{\mathrm{HT}}_{\mathrm{tor}}$ et $M'\subset M$, on appelle {\it adhérence schématique de $M'$} le sous-object strict 
$$
(M', \O_C. \a (M'),\a_{|M'}).
$$
Enfin, pour $X=(T,\omega,\a)\in \M^{\mathrm{HT}}_{\Zp}$ et $n\geq 1$, on note
$$
X[p^n] := (p^{-n}T/ T, p^{-n}\omega/ \omega, \a )\in \M^{\mathrm{HT}}_{\mathrm{tor}}.
$$

\subsection{Adaptation des résultats sur les groupes $p$-divisibles} 

Comme dans le théorème \ref{SFKEGI39459sguzSF20refq}, pour $X\in \M^{\mathrm{HT}}_{\Zp}$, on vérifie que 
$$
\HN (X):= \underset{n\drt +\infty}{\lim} \frac{1}{n} \HN ( X[p^n]) (n\bullet)
$$
définit une fonction concave appelée polygone de Harder-Narasimhan renormalisé de $X$.
On a alors la proposition suivante qui est une adaptation immédiate du théorème \ref{theo:structure groupes type HN} et de la proposition \ref{KdfjudsfEITfhsfroodju23459sedojour}.

\begin{prop} \label{prop:sur essentiel}
Pour $X\in \M^{\mathrm{HT}}_{\Zp}$:
\begin{enumerate}
\item $X[p]$ est semi-stable si et seulement si pour tout $n\geq 1$, $X[p^n]$ l'est. On dit alors que $X$ est semi-stable.
\item Si $X$ est semi-stable il est alors semi-stable à isogénie près, i.e. $X\unp$ est semi-stable dans $\M^{\HT}_{\Qp}$.
\item On a $\HN (X[p]) =\HN (X)$ si et seulement si $X$ possède une filtration croissante $(X_i)_{0\leq i\leq r}$ avec $X_0=0$, $X_r=X$, $X_{i+1}/X_i$ semi-stable et la suite $\big ( \mu (X_{i+1}/X_i)\big )_{0\leq i \leq r-1}$ est strictement décroissante. On dit alors que $X$ est de type HN.
\item Si $X$ est isogène à un module de Hodge-Tate entier de type HN, alors la filtration de Harder-Narasimhan de $X\unp$ dans $\M^{\HT}_{\Qp}$ est donnée par la filtration de Harder-Narasimhan entière précédente et $\HN (X)=\HN ( X\unp)$.
\end{enumerate}
\end{prop}

On va en fait voir (théo. \ref{theo:essentiel}) que tout $X\in \M^{\HT}_{\Zp}$ est isogène à un module entier de type HN, alors qu'à priori ce résultat n'est vrai que lorsque la valuation du corps de base est discrète pour les groupes $p$-divisibles.

\subsection{Le résultat clef sur les modules de Hodge-Tate entiers}

On va reprendre l'algorithme de descente de la section \ref{KDSalgoGisdgETP348hhdidncbdsyg} et montrer qu'il converge toujours pour les modules de Hodge-Tate entiers.

\begin{theo}\label{theo:essentiel}
Tout $X\in \M^{\HT}_{\Zp}$ est isogène à un module de Hodge-Tate entier de type HN.
\end{theo}
\begin{proof}
Notons $X=(T,\omega,\a)$. Soit, pour tout entier $k\geq 1$,
$$
Y_k =\text{le premier cran de la filtration de Harder-Narasimhan de }X[p^k].
$$
Comme dans le lemme \ref{KQIZRozru3485gouroulemmedescente} on a:
\begin{itemize}
\item la suite $(Y_k)_{k \geq 1}$ forme une suite croissante de modules de Hodge-Tate de torsion semi-stables de même pente,
\item pour $i\geq j\geq 1$, $Y_i[p^j]\subset Y_j$,
\end{itemize}
en utilisant  que la catégorie des modules de Hodge-Tate de torsion semi-stables de pente fixée est abélienne.  Notons 
$$
\mu_{\mathrm{max}} (X)
$$
la pente commune des $(Y_k)_{k\geq 1}$. On note $\La_k$ le réseau associé à $Y_k$, $T\subset \La_k\subset T\otimes \Qp$. Ils forment une suite croissante de réseaux
$$
T\subset \La_1\subset \cdots \subset \La_k \subset \La_{k+1} \subset \cdots
$$
telle que 
$$
\La_k\subset p^{-k} T, \ \ \La_{k+1}\cap p^{-k}\La \subset \La_k\ \text{ et } \ p\La_{k+1}\subset \La_k.
$$
Posons alors 
$$
\La_{\infty} = \bigcup_{k\geq 1} \La_k
$$
qui est donc un sous-$\Zp$-module de $\La\otimes \Qp$ tel que $\La_{\infty}\cap p^{-k} T=\La_k$.
On a $\La_\infty = T\otimes \Qp$ si et seulement si $X$ est semi-stable. On suppose désormais que ce n'est pas le cas. Comme dans le lemme \ref{kourutopresakitekbouibui87814} deux possibilités s'offrent alors:
\begin{enumerate}
\item soit $\La_\infty /T$ est fini, c'est à dire qu'il existe $k_0\geq 1$ tel que pour $k\geq k_0$ on ait $\La_k=\La_{k_0}$,
\item soit $\La_\infty /T$ est infini, auquel cas $X\unp$ possède un un sous-objet non nul strict dans $\M_{\Qp}^{\HT}$ 
$$
A\unp \subsetneq X\unp, 
$$
avec $A\in \M^{\HT}_{\Zp}$ semi-stable et $\mu ( A\unp)=\mu_{\mathrm{max}} (X)>\mu (X\unp)$.
\end{enumerate}
Supposons donc maintenant que $X\unp$ est stable dans $\M^{\HT}_{\Qp}$ mais qu'il n'est pas isogène à un module de Hodge-Tate entier semi-stable. Il existe donc une suite infinie d'isogénies 
$$
X=X_0\xrig{\ f_0\ } X_1\xrig{\ f_1\ }  \dots\xrig{\ f_{i-1}\ }
X_i\xrig{\ f_i\ } X_{i+1} \xrig{\ f_{i+1}\ } \dots
$$
telle que pour tout $i$, $\ker f_i\neq 0$ et 
$$
\mu (\ker f_i)=\mu_{\mathrm{max}} (X_i)
$$
avec 
$$
\mu_{\mathrm{max}} (X_0) >\mu_{\mathrm{max}} (X_1) >\cdots > \mu_{\mathrm{max}} (X_i) > \mu_{\mathrm{max}} (X_{i+1}) >\cdots .
$$
Notons $\mu_i= \mu_{\mathrm{max}} (X_i)$ et 
$$
\mu_\infty = \underset{i\drt +\infty}{\lim} \mu_i.
$$
On a donc $\mu_\infty \geq \mu (X\unp)$. À la suite d'isogénies précédente est associée une suite strictement croissante de réseaux 
$$
T =L_0\subsetneq L_1\subsetneq \dots \subsetneq L_i\subsetneq
L_{i+1}\subsetneq \dots.
$$
Soient
$$
L_\infty =\bigcup_{i\geq 0} L_i
$$
et
$$
(L_\infty)_{div}
$$
sa partie divisible, un sous-$\Qp$-espace vectoriel de $T\otimes
\Qp$. Supposons par l'absurde que $$(L_\infty)_{div}\neq 0.$$
Le $\Zp$-module $(L_\infty)/(L_\infty)_{div}$ est de type
fini (cf. lemme \ref{lemme:stupide stupide}). Or
$$
L_\infty/ (L_\infty)_{div} = \bigcup_{i\geq 0} L_i + (L_\infty)_{div}/(L_\infty)_{div}
$$
et par noethérianité de $L_\infty/(L_\infty)_{div}$ on en déduit qu'il
existe $i_0\geq 0$ tel que pour $i\geq i_0$ on ait
$$
L_i+(L_\infty)_{div} =L_{i_0}+L_\infty.
$$
Quitte à remplacer $X$ par $X_{i_0}$ on
peut supposer que pour $i\geq 0$ l'égalité précédente est vérifiée i.e.
\begin{equation}\label{eq:egalite en infini}
\forall i\geq 0, \ \ L_i+ (L_\infty)_{div} = L_\infty.
\end{equation}
Notons maintenant
\begin{eqnarray*}
T' &=& (L_\infty)_{div}\cap T \\
X' &=& (T', \O_C.\a (T'), \a_{|T'})
\end{eqnarray*} 
et pour $i\geq 0$, 
$$
N_i = \left ((L_\infty)_{div}\cap L_i \right )/ T' \subset  T'\otimes \Qp/\Zp.
$$
On a donc 
$$
T'\otimes \Qp/\Zp  =\bigcup_{i\geq 0} N_i.
$$
Il y a un morphisme naturel
$$
f:X'\ldrt X
$$
qui satisfait aux hypothèses du lemme \ref{lemme:bornes lemme clef}. On en déduit
l'existence d'une constante $C$ telle que pour $i\geq 0$ on ait
$$
\left |\deg \left ( \overline{N_i}^{X'}\right )-\deg\left
    (\overline{N_i}^{X}\right )\right |\leq C.
$$
De plus, comme conséquence de l'égalité (\ref{eq:egalite en infini}), 
$$
\overline{N_i}^{X} = \ker (X_0\drt \dots \drt X_i).
$$
Maintenant, étant donné que $\bigcup_i N_i = T'\otimes \Qp/\Zp$,
pour tout entier positif $r$ il existe $i(r)\geq 0$ tel que 
$$
p^{-r} T'/T'\subset N_{i(r)}.
$$
Cette inclusion implique que 
$$
\deg \left ( \overline{N_{i(r)}}^{X'} \right ) =\deg \left (
  \overline{p^r N_{i(r)}}^{X'} \right ) + r \deg \big ( X'\unp \big ).
$$
On a donc pour $r\geq 1$, par application du lemme \ref{lemme:bornes lemme clef}, 
\begin{equation}\label{eq: deux C}
\left |  \deg\left  ( \overline{p^r N_{i(r)}}^{X} \right ) +r\deg
  \big (X'\unp\big ) - \deg
  \left (\overline{N_{i(r)}}^X \right ) \right | \leq 2 C.
\end{equation}
Maintenant, étant donné que $p^r N_{i(r)}\subset N_{i(r)}$ et que 
le point $$\left (\mathrm{ht} \left  (
  \overline{p^r N_{i(r)}}^{X} \right ), \deg \left  ( \overline{p^r
    N_{i(r)}}^{X} \right ) \right )$$ 
    est en dessous du polygone de
Harder-Narasimhan de $\overline{N_{i(r)}}^X$, on a 
\begin{eqnarray}\label{eq;deux deux trois}
\deg \left  ( \overline{p^r N_{i(r)}}^{X} \right ) &\leq & \deg
\overline{N_{i(r)}}^X -\mu_\infty\, \mathrm{long}_{\Zp} \left ( N_{i(r)}/p^r
  N_{i(r)}\right ) \\
  \nonumber
&=&  \deg
\overline{N_{i(r)}}^X - r \,\mu_\infty \,\mathrm{ht} (X'\unp).
\end{eqnarray}
Combinant les deux dernière inégalités (\ref{eq: deux C}) et (\ref{eq;deux deux trois}) on obtient 
$$
 \deg
  \overline{N_{i(r)}}^X - r\,\deg (X'\unp) -2C \leq \deg 
\overline{N_{i(r)}}^X - r \,\mu_\infty \,\mathrm{ht} (X'\unp).
$$
Cela implique que 
$$
\mu_{\infty} \leq \mu (X'\unp) +\frac{2C}{r \ \htt (X'\unp)}
$$
et en faisant $r \drt +\infty$ on obtient
$$
\mu_{\infty}\leq \mu (X'\unp),
$$
ce qui implique que 
$$
\mu (X\unp) \leq \mu (X'\unp).
$$
Cela est en contradiction avec la stabilité de $X\unp$. On en déduit que $(L_\infty)_{div}=0$, 
ce qui nous permet de conclure que \og l'algorithme de descente converge en temps fini\fg{}.
\end{proof}

\begin{lemm}[cf. \cite{Boulbaki} A VII.58, exo. 6] \label{lemme:stupide stupide}
Tout sous-$\Zp$-module d'un $\Qp$-espace vectoriel de dimension finie est isomorphe à la somme directe d'un $\Zp$-module libre de rang fini et d'un $\Qp$-espace vectoriel de dimension finie.
\end{lemm}
\begin{proof}
Cela résulte essentiellement, par récurrence sur la dimension du $\Qp$-espace vectoriel considéré, de l'annulation 
$$
\Ext^1_{\Zp} (\Qp,\Zp)=0.
$$
En écrivant $\Qp=\bigcup_{n\geq 0} p^{-n}\Zp$, cela se dévisse en l'annulation 
\begin{equation*}
R^1\underset{n\geq 0}{\limp} \Hom (p^{-n}\Zp,\Zp) = R^1\underset{n\geq 0}{\limp} p^n \Zp = 0 \qedhere
\end{equation*}
\end{proof}

\begin{rema}
Le lemme \ref{lemme:stupide stupide} est faux pour des modules sur un anneau de valuation discrète non complet. C'est par exemple le cas de $\Z_{(p)}$, puisque $R^1\underset{n\geq 0}{\limp} p^n \Z_{(p)}\neq 0$. Il existe un  $\Z_{(p)}$-module $M$ extension non scindée de $\Q$ par $\Z_{(p)}$. Ce n'est pas un $\Z_{(p)}$-module de type fini, mais on a $M_{div}=0$.
\end{rema}

\begin{lemm}\label{lemme:bornes lemme clef}
Soit $X'\drt X$ un morphisme dans $\mathcal{M}^{HT}_{\Zp}$ tel que 
\begin{enumerate}
\item 
$T(X')$ soit facteur direct dans $T(X)$ et induit donc une injection 
$$
\{\text{sous-groupes finis de } T(X')\otimes \Qp/\Zp \} \hookrightarrow
\{\text{sous-groupes finis de } T(X)\otimes \Qp/\Zp \},
$$
\item on ait $\omega_{X'}\hookrightarrow \omega_{X}$ i.e. est injectif.
\end{enumerate}
 Il existe
alors une constante $C$ telle que pour tout sous-groupe fini $N$ de
$T(X')\otimes \Qp/\Zp$, on ait
$$
 \deg\left (
    \overline{N}^{X}\right ) -C\leq 
 \deg \left (\overline{N}^{X'}\right ) \leq \deg\left (
    \overline{N}^{X}\right )
$$
où $\overline{N}^X$, resp. $\overline{N}^{X'}$, désigne l'adhérence schématique dans $X$, resp. $X'$.
\end{lemm}
\begin{proof}
Notons $X=(T,\omega,\a)$ et $X'=(T',\omega',\a')$. 
Il y a un diagramme
$$
\xymatrix@R=6mm{
T' \ar@{^(->}[r] \ar[d]_{\a'} & T \ar[d]^\a \\
\omega'\ar@{^(->}[r] & \omega
}
$$
Notons $\omega''=\omega \cap \left (\omega'\otimes C\right ) $. Il y a une factorisation
$$
\omega'\hookrightarrow \omega''\hookrightarrow \omega
$$
avec $\omega''$ facteur direct dans $\omega$. Soit $c$ un entier tel
que $\omega''/\omega'$ soit annulé par $p^c$. Pour un entier $n\geq c$,
il y a une suite exacte
$$
0\ldrt \omega''/\omega'\ldrt p^{-n}\omega'/\omega'\xrig{\;u_n\;}
p^{-n}\omega''/\omega''\ldrt \omega''/\omega'\ldrt 0
$$
et une inclusion
$$
i_n: p^{-n}\omega''/\omega'' \hookrightarrow p^{-n}\omega/\omega.
$$
Maintenant, si $N=L/T'$ avec $T'\subset L \subset p^{-n}T'$ et
$n\geq c$, $\overline{N}^{X'[p^n]}$ est donné par le diagramme
$$
L/T'\xrig{\;\a'\;} \O_C.\a'(L)/\omega',
$$
tandis que $\overline{N}^{X[p^n]}$ est donné par le composé
$$
L/T'\xrig{\;\a'\;} \O_C.\a' (L)/\omega' \xrig{\;u_n\;} u_n \left (
  \O_C.\a'(L)/\omega'\right ) \xrig{\ \sim \ } i_n\circ u_n \left (
  \O_C.\a'(L)/\omega'\right ).
$$
Or, il y a une suite exacte
$$
0\ldrt \ker u_n \cap \left ( \O_C.\a'(L)/\omega'\right )\ldrt
\O_C.\a'(L)/\omega' \ldrt  u_n \left (
  \O_C.\a'(L)/\omega'\right )\ldrt 0
$$
avec $\ker u_n \cap \left ( \O_C.\a'(L)/\omega'\right ) \subset
\omega''/\omega'$. Si $C=v(\mathrm{Fitt}_0 \ \omega''/\omega' )$ on a
donc
$$
v\left ( \mathrm{Fitt}_0\
\O_C.\a'(L)/\omega' \right ) -C \leq 
v\left (\mathrm{Fitt}_0 \ u_n \left (
  \O_C.\a'(L)/\omega'\right )\right )\leq v\left (\mathrm{Fitt}_0\
\O_C.\a'(L)/\omega' \right )
$$
d'où l'inégalité annoncée. 
\end{proof}

\section{Retour aux groupes $p$-divisibles: deux résultats clef sur les polygones de HN renormalisés}
\label{sec:Retour aux gpdiv}

Dans cette section le corps $K|\Qp$ est quelconque. On note $C=\widehat{\overline{K}}$. 

\subsection{Le polygone renormalisé est un polygone}

Rappelons le résultat suivant: si $G$ est un groupe plat fini sur $\O_C$ alors le conoyau
de 
$$
\a_G\otimes 1: G(\O_C)\otimes \O_C\ldrt \omega_{G^D}
$$
est annulé par $p^{1/(p-1)}$ si $p\neq 2$ et $4$ si $p=2$ (\cite{LivreIso} théo. II.1.1, \cite{FarguesCanonique} théo. 3). C'est cette estimée qui a inspiré à l'auteur l'introduction des filtration de Harder-Narasimhan des modules de Hodge-Tate.

\begin{theo}\label{theo:poly renormalise est un poly}
Pour $H\in \BT_{\O_K}$:
\begin{enumerate}
\item Le polygone renormalisé $\HN (H)$ est un polygone à abscisses et ordonnées de rupture entières.
\item  Plus précisément, $\HN (H)$ coïncide avec le polygone de Harder-Narasimhan de son application de Hodge-Tate $V_p(H)\xrightarrow{\ \a_H\ } \omega_{H^D}\otimes C$ dans $\M^{\HT}_{\Qp}$.
\end{enumerate}
\end{theo}
\begin{proof}
On suppose $p\neq 2$, le cas $p=2$ étant laissé au lecteur. 
Pour tout sous-groupe plat fini $G$ de $H[p^n]$, puisque $\omega_{H^D}/p^n\omega_{H^D} \twoheadrightarrow \omega_{G^D}$, $\omega_{G^D}$ est engendré par $\dim H^D$-éléments. Couplé avec le fait que $\mathrm{coker} (\a_G\otimes 1)$ est annulé par $p^{1/(p-1)}$, on en déduit que 
$$
\deg \omega_{G^D} - \frac{\dim H^D}{p-1}\leq   \deg \mathrm{coker} (\a_G\otimes 1) \leq \deg \omega_{G^D}.
$$
Soit $X=(T_p(H),\mathrm{Im} (\a_H\otimes 1),\a_H)\in \M^{\HT}_{\Zp}$. On a donc 
$$
\HN (H[p^n]) \leq \HN (X[p^n]) \leq \HN (H[p^n]) + \frac{\dim H^D}{p-1}.
$$
En divisant par $n$ et en appliquant le procédé de renormalisation de la définition \ref{SFKEGI39459sguzSF20refq}, on obtient le résultat grâce à la proposition \ref{prop:sur essentiel} et au théorème \ref{theo:essentiel}.
\end{proof}

\begin{rema}
Le théorème précédent ne dit pas que $H$ est isogène à un groupe de type HN. \`A priori si la valuation de $K$ n'est pas discrète il n'y a pas de raison pour que ce soit le cas, c'est seulement le cas pour son module de Hodge-Tate entier (cf. néanmoins théo. \ref{theo:theo principal ou pas}). 
\end{rema}

\begin{rema}
D'après \cite{ScholzeWeinstein}, lorsque $K=C$, la catégorie $\M^{\HT}_{\Zp}$ est équivalente à $\BT_{\O_C}$. Néanmoins, il faut faire attention à ce qu'il ne s'agit pas d'une équivalence de catégories exactes. Une suite $0\drt H_1\drt H_2\drt H_3\drt 0$ dans $\BT_{\O_C}$,  qui donne lieu à une suite exacte dans $\M^{\HT}_{\Zp}$, est exacte si et seulement si la suite $0\drt \omega_{H_1^D}\drt \omega_{H_2^D}\drt \omega_{H_3^D}\drt 0$ est exacte. Or, dans $\M_{\Zp}^{\HT}$ on remplace $\omega_{H^D}$ par l'image de $\a_H\otimes 1$. 
\end{rema}

\subsection{Inégalité entre Newton et Harder-Narasimhan}
\subsubsection{Polygone de HN des modifications admissibles de fibrés sur la courbe}

Soit $X=X_{C^\flat}$ la courbe schématique munie de son point à l'infini $\infty\in |X|$ 
de corps résiduel $C$ (\cite{Courbe}). Considérons la catégorie exacte des modifications admissibles effectives  de fibrés sur la courbe (cf. \cite{ConfLaumon}),
$$
\mathrm{Modif}^{ \mathrm{ad},\geq 0} = \{ (\E_1,\E_2, u)\}
$$
où 
\begin{itemize}
\item $\E_1$ est un fibré semi-stable de pente $0$,
\item $\E_2$ est un fibré sur la courbe,
\item $u:\E_1\hookrightarrow \E_2$ est une modification supportée en $\infty$.
\end{itemize}

On définit des fonctions degré et rang sur cette catégorie en posant
\begin{eqnarray*}
\deg (\E_1,\E_2, u) &=& \mathrm{long} (\mathrm{coker} u ) \\
\mathrm{rg} (\E_1,\E_2, u) &=& \mathrm{rg} \,\E_1.
\end{eqnarray*}
On note $\mu=\frac{\deg}{\mathrm{rg}}$ la fonction pente associée. On dispose d'un \og foncteur fibre générique\fg{} sur $\mathrm{Modif}^{ \mathrm{ad},\geq 0}$ 
$$
(\E_1,\E_2,u) \longmapsto H^0(X,\E_1)
$$
à valeurs dans les $\Qp$-espaces vectoriels de dimension finie. Plus précisément, ce foncteur induit une bijection entre les sous-objets stricts de $(\E_1,\E_2,u)$ et les sous-espaces vectoriels de $H^0(X,\E_1)$. De plus, si un morphisme
$$
(\E_1,\E_2,u) \ldrt (\E'_1,\E'_2,u')
$$
induit un isomorphisme $H^0(X,\E_1)\xrig{\sim} H^0(X,\E'_1)$, i.e. de façon équivalente un isomorphisme $\E_1\xrig{\sim}\E'_1$, on a alors un diagramme
$$
\xymatrix@C=8mm@R=8mm{
0\ar[r] &  \E_1\ar[d]^{\simeq}\ar[r] & \E_2 \ar@{^(->}[d]\ar[r] & \mathrm{coker} (u) \ar[d]\ar[r] & 0 \\
0\ar[r] & \E'_1\ar[r] & \E'_2\ar[r] & \mathrm{coker} (u')\ar[r] & 0
}
$$
où $\E_2\drt \E'_2$ est un monomorphisme, puisque c'est le cas génériquement (et c'est donc une modification de fibrés). On en déduit que $\mathrm{coker} (u)\drt \mathrm{coker} (u')$ est un monomorphisme et que 
$$
\deg (\E_1',\E_2',u') =  \deg (\E_1,\E_2,u) + \mathrm{long}\ \mathrm{coker} \big ( \mathrm{coker} (u)\hookrightarrow \mathrm{coker} (u') \big ).
$$
On a  donc $\deg (\E_1',\E_2',u') \geq  \deg (\E_1,\E_2,u)$ avec égalité si et seulement si le morphisme est un isomorphisme. On dispose donc de filtrations de Harder-Narasimhan dans $\mathrm{Modif}^{ \mathrm{ad},\geq 0}$. La proposition suivante est le point clef que nous utiliserons.

\begin{prop}\label{prop:inegalite poly HN de modif et Newton}
Pour $(\E_1,\E_2,u)\in \mathrm{Modif}^{ \mathrm{ad},\geq 0}$ on a 
$$
\HN (\E_1,\E_2,u) \leq \HN (\E_2).
$$
\end{prop}
\begin{proof}
Il s'agit d'une simple conséquence de ce que, pour $(\E'_1,\E'_2,u')$ un sous-objet strict
de $(\E_1,\E_2,u)$, on a 
$$
\deg (\E'_1,\E'_2,u') =\deg (\E'_2) \leq \HN (\E_2) \big ( \mathrm{rg} \,\E'_2\big )
$$
puisque $\E'_2\hookrightarrow \E_2$.
\end{proof}

Il y a un foncteur exacte pleinement fidèle
\begin{equation}\label{eq:foncteur modules filtres vers modifications}
\VectFil_{C/\Qp}^{[0,1]} \ldrt \mathrm{Modif}^{ \mathrm{ad},\geq 0}
\end{equation}
d'image la catégorie des modifications \og minuscules\fg{}, où $\VectFil_{C/\Qp}^{[0,1]}$  est la sous-catégorie pleine de $\VectFil_{C/\Qp}$ formée des couples $(V,\Fil^\bullet V_C)$ tels que 
$$
\begin{cases}
\Fil^0 V_C=V_C \\
\Fil^2 V_C = (0).
\end{cases}
$$
À $(V,\Fil^\bullet V_C)$ on associe la modification $(\E_1,\E_2,u)$ telle que 
\begin{itemize}
\item $\E_1= V\otimes_{\Qp} \O_X$
\item $\widehat{\E}_{2,\infty} / \widehat{\E}_{1,\infty} =t^{-1} \Fil^1 V_C \subset V\otimes B_{dR}/B^+_{dR}$.
\end{itemize}

Les fonctions pentes se correspondent via ce foncteur et la filtration de Harder-Narasimhan de $(\E_1,\E_2,u)$ est induite par celle de $(V,\Fil^\bullet V_C)$. 
\\

Rappelons maintenant (\cite{Courbe} chap. 8 et \cite{ScholzeWeinstein} prop. 5.1.6) que si $H$ est un groupe $p$-divisible sur $\O_C$  de fibre spéciale $H_{k_C}$, alors via le foncteur (\ref{eq:foncteur modules filtres vers modifications}) composé avec le foncteur
$$
\M^{\HT}_{\Qp} \ldrt \VectFil^{[-1,0]} 
$$
de la section \ref{sec:quelques definitions}, 
on a
$$
(V_p(H), \omega_{H^D}\unp, \a_H) \longmapsto (V_p (H)\otimes_{\Qp} \O_X, \E (\mathbb{D} (H_{k_C}),p^{-1}\ph ), u),
$$
où $u$ est un morphisme de comparaison de périodes cristallines de Fontaine (il s'agit essentiellement de la traduction en termes de fibrés vectoriels des théorèmes de comparaison de Fontaine pour les groupes $p$-divisibles) et $(\mathbb{D} (H_{k_C}),\ph)$ est le module de Dieudonné covariant de $H_{k_C}$. De cela, de la proposition \ref{prop:inegalite poly HN de modif et Newton} et du théorème \ref{theo:poly renormalise est un poly}, on déduit le résultat suivant.

\begin{theo}\label{theo:inegalite Newton et HN}
Pour $H\in \BT_{\O_K}$ on a 
$$
\HN (H)\leq \Newt (H_k)^{\diamond}.
$$
En particulier, si la fibre spéciale $H_k$ est isocline, $\HN(H)$ est une droite de pente $\frac{\dim H}{\htt H}$.
\end{theo}

\section{Application aux espaces de modules de groupes $p$-divisibles}
\label{sec:application aux espaces de modules}

\subsection{Déstabilisation par une isogénie des groupes $p$-divisibles semi-stables}

Soit $K|\Qp$ quelconque comme précédemment et $H\in \BT_{\O_K}$.

\begin{prop}\label{prop:destabilisation}
Supposons $H$ semi-stable et soit $f:H\drt H'$ une isogénie qui n'est pas un isomorphisme. Sont équivalents:
\begin{enumerate}
\item $H'$ est semi-stable,
\item $\ker f$ est un groupe plat fini semi-stable de pente $\mu (H)$.
\end{enumerate}
\end{prop}
\begin{proof}
Notons $G=\ker f$. 
Supposons $H'$ semi-stable. Soit $n\geq 1$ tel que $G\subsetneq H[p^n]$. Puisque 
$H[p^n]/G\subset H'$, on a par semi-stabilité de $H'$
$$
\mu (H[p^n]/G)\leq \mu (H')=\mu(H).
$$
De la suite exacte 
$$
0\ldrt G\ldrt H[p^n]\ldrt H[p^n]/G\ldrt 0,
$$
on déduit que si $\mu(G)<\mu (H)$ alors $\mu (H[p^n])<\mu (H)$. Cela est impossible puisque $\mu(H[p^n])=\mu (H)$. On en déduit que $\mu(G)=\mu(H)$ et $G$ est donc semi-stable.

Réciproquement, si $\mu(G)=\mu (H)$, il y a une suite exacte 
$$
0\ldrt H[p]\ldrt p^{-1}G\xrig{\ \times p\ } G\ldrt 0.
$$
Celle-ci implique que $p^{-1} G$ est semi-stable de pente $\mu (H)$. De la suite exacte 
$$
0\ldrt G\ldrt p^{-1} G \ldrt H'[p]\ldrt 0
$$
on déduit alors que $H'[p]$ est semi-stable.
\end{proof}

Notons le corollaire évident suivant mais qui aura une interprétation géométrique agréable.

\begin{coro}
Si $H[p]$ est stable alors pour tout isogénie $H\drt H'$ ne se factorisant pas par la multiplication par $p$, $H'$ n'est pas semi-stable.
\end{coro}

La proposition suivante sera fort utile dans la suite.

\begin{prop}\label{prop:astuce BT tronque 1}
Soit $H\in \BT_{\O_K}$ semi-stable tel que $(\dim H,\htt\ H)=1$ et $G\subset H$ un sous-groupe plat fini non nul vérifiant
\begin{enumerate}
\item $H[p]\not\subset G$,
\item $G$ est semi-stable de pente $\mu(H)$.
\end{enumerate}
Alors, si $h=\htt \, H$, $$G\subset H[p^{h}]$$ et pour $1\leq k\leq h-1$, $\htt( p^k G )\leq h-k$. 
\end{prop}
\begin{proof}
Puisque $G$ est semi-stable de pente $\mu(H)$, pour tout $i,j\geq 0$, $p^i G[p^j]$ est semi-stable de pente $\mu (H)$ lorsqu'il est non nul. On va travailler dans la catégorie abélienne des groupes semi-stables de pente fixée $\mu (H)$, catégorie dans laquelle tout épimorphisme en fibre générique est plat surjectif. 
Supposons par l'absurde que $pG[p^2] =G[p]$. La suite exacte 
$$
0\ldrt G[p]\ldrt G[p^2]\xrig{\ \times p\ } G[p]\ldrt 0
$$
montre alors que $G$ est groupe de Barsotti-Tate tronqué d'échelon $2$. Il en résulte que $\deg (G[p])\in \N$. Mais puisque $H[p]\not\subset  G$, on a $G[p]\not\subset H[p]$ et donc l'égalité 
$\mu (H)= \mu (G[p])$ est impossible puisque $(\dim  H, \htt \ H)=1$. On a donc $pG[p^2]\subsetneq G[p]$. 

Puisque $pG[p^2]\subsetneq G[p]$, $$G[p^2]/G[p]=(G/G[p])[p]\subsetneq (H/G[p])[p].$$
On peut alors applique de nouveau le raisonnement précédent pour conclure que, si $G[p^2]\neq G[p]$, alors $p G[p^3]\subsetneq G[p^2]$. On conclut ainsi aisément par récurrence.
\end{proof}

De la démonstration précédente on tire le résultat suivant.

\begin{prop}\label{prop:desta clef deux suite retour}
Soit $H\in \BT_{\O_K}$ et $G\subset H$ un sous-groupe plat fini semi-stable tel que $H[p]\subsetneq G$ et 
$$
\mu (G)\notin \Big \{ \frac{d'}{h'}\ | \ d',h'\in \N, 0\leq d'\leq h' \leq h-1 \Big \}.
$$
Alors, $G\subset H [p^{\htt H}]$ et pour $1\leq k\leq h-1$, $\htt( p^k G )\leq h-k$. 
\end{prop}

Notons également le résultat suivant conséquence des propositions \ref{exisutnidupsemistabe35129643} et \ref{prop:destabilisation}.

\begin{prop}
Supposons $H$ de type HN et soit $f:H\drt H'$ une isogénie qui n'est pas un isomorphisme.
Sont équivalents:
\begin{enumerate}
\item $H'$ est de type HN
\item les pentes de $\HN (\ker f)$ sont contenues dans les pentes de $\HN(H)$. 
\end{enumerate}
\end{prop}

\subsection{Le théorème de réduction}

Dans cette section {\it $K$ est quelconque}. Voici une des principales applications géométriques des résultats précédents. 

\begin{theo}\label{theo:theo principal ou pas}
Soit $H\in \BT_{\O_K}$  tel que $(\dim H,\htt \ H)=1$. 
\begin{enumerate}
\item Si $\HN (H)$ est une droite alors $H$ est isogène à un groupe $p$-divisible semi-stable.
\item Si la fibre spéciale de $H$ est isocline alors $H$ est isogène à un groupe $p$-divisible semi-stable.
\end{enumerate}
\end{theo}
\begin{proof}
D'après le théorème \ref{theo:inegalite Newton et HN} le point (1) implique le point (2). 
On va montrer que l'algorithme de descente de la section \ref{KDSalgoGisdgETP348hhdidncbdsyg} s'arrête en temps fini. On peut alors supposer que $K=C$ est algébriquement clos. D'après le théorème 5.1.4 de \cite{ScholzeWeinstein} on peut supposer que $H$ est donné par un $\O_C$-point de l'espace de Rapoport-Zink $\M$ sur $\spf ( \breve{\Z}_p)$ des déformations par isogénies d'un groupe $p$-divisible  sur $\Fpb$. On note $\M_\eta$ sa fibre générique comme espace de Berkovich.
Pour simplifier les notations on note encore $H$ pour le groupe $p$-divisible universel sur $\M$, le groupe dont on est parti étant maintenant une spécialisation en un $\O_C$-point de $H$.
 Le lieu de semi-stabilité de $H$ est un domaine analytique fermé 
$$
\M_\eta^{ss} \subset \M_\eta.
$$
Soit 
$$
U=\{ x\in \M_\eta \ |\ \HN (H^{\mathscr{u}}_x)\text{ est une droite }\},
$$
un ouvert de $\M_\eta$ (prop. \ref{DSQDFKqsfuetD35xezmicontienutzskdfuz} et point (1) du théo. \ref{theo:poly renormalise est un poly}) qui contient $\M_\eta^{ss}$.
L'algorithme de descente de la section \ref{KDSalgoGisdgETP348hhdidncbdsyg} définit une application 
$$
T:|U|\drt |U|
$$
définie par
\begin{itemize}
\item $T(x)=x$ si $H^{\mathrm{u}}_x$ est semi-stable,
\item sinon, $T(x)$ est donné par l'isogénie $H_x\drt H_x /G$, où $G$ est le plus grand sous-schéma en groupes plat fini de $H_x$ semi-stable de pente $\mu_{\mathrm{max}} (H_x)$.
\end{itemize}
L'algorithme de descente le long de l'orbite de Hecke de $x$ est alors donné par $(T^{n}(x))_{n\geq 0}$ et il s'agit de montrer que pour $n\gg 0$, $T^{n} (x)$ est un point fixe de $T$.

Pour cela on globalise la situation. D'après le théorème d'uniformisation de Rapoport-Zink (\cite{RZ} théo. 6.23), il existe un sous-groupe discret  $\Gamma\subset J(\Qp)$ tel que 
$$
\GG\bc \M\hookrightarrow \mathscr{S}_\eta,
$$
où $\mathscr{S}$ est le complété $p$-adique d'un modèle entier d'une variété de Shimura de type PEL non-ramifiée sur $\breve{\Z}_p$. On dispose de correspondances de Hecke sur $\mathscr{S}_\eta$ et on se ramène ainsi à démontrer notre théorème pour $\mathscr{S}_\eta$, l'avantage étant que $|\mathscr{S}_\eta|$ est compact. Si $\mathscr{A}$ est le schéma abélien universel sur $\mathscr{S}$ on note encore $H=\mathscr{A}[p^\infty]$.
Soit donc $\mathscr{S}_\eta^{ss}\subset \mathscr{S}_\eta$ le lieu de semi-stabilité des points de $H[p]$, un domaine analytique compact. 
On note encore $U\subset \mathscr{S}_\eta$ l'ouvert où le polygone de Harder-Narasimhan renormalisé est une droite,
$$
\mathscr{S}_\eta^{ss}\subset U \subset \mathscr{S}_\eta.
$$
Notons 
$$
\mu_{\mathrm{max}}: |U|\ldrt [d/h,1],
$$
$d=\dim H$, $h=\htt \, H$, 
la fonctions continue donnée par la plus grande pente du polygone de Harder-Narasimhan des points de $p$-torsion. On a donc 
$$
|\mathscr{S}_\eta^{ss}| =\mu_{\mathrm{max}}^{-1} (d/h).
$$
Il y a une action par correspondances des correspondances de Hecke en $p$ sur $\mathscr{S}_\eta$. Elles sont en bijection avec
$$
\GL_h (\Zp) \bc \GL_h (\Qp)/ \GL_h (\Zp).\Qp^\times 
$$
(le facteur $\Qp^\times$ agit trivialement sur l'espace, pas sur les fibrés automorphes dont nous n'avons pas besoin ici).
On identifie cet ensemble de correspondances à $(\Z^h)^+/\Z$, les suites décroissantes de $\Z^h$ modulo translations. Pour $\underline{a}=(a_1,\dots,a_h)\in (\Z^h)^+$ avec $a_1\geq \cdots \geq a_h$ on note ${Hecke}_{\underline{a}}$ la correspondance de Hecke associée. D'après la proposition \ref{prop:astuce BT tronque 1}
$$
Y:= \bigcup_{2h\geq a_1\geq \cdots \geq a_h\geq 0 \atop a_1-a_h>h} Hecke_{\underline{a}} (\mathscr{S}_\eta^{ss})
$$
vérifie 
$$
Y\cap \mathscr{S}_\eta^{ss} =\emptyset.
$$
Par compacité de $Y$, 
$$
\a:= \inf_Y \mu_{\mathrm{max}}>\frac{d}{h}.
$$
Notons maintenant 
$$
\l:= \inf \{\a\} \cup \Big \{ \frac{d'}{h'}\ \Big |\ d',h'\in \N, \ 1\leq d'\leq h'\leq h-1, \frac{d'}{h'}>\frac{d}{h}\Big \}.
$$
On a $\l>\frac{d}{h}$. Soit $x\in U$ tel que $$\mu_{\mathrm{max}} (H_x)<\l.$$ 
Supposons que $K(x)$ est de valuation discrète et donc (prop. \ref{SDKGDSHGIERTPO3268SRsdgsee}) $T^n(x)\in \mathscr{S}_\eta^{ss}$ pour $n\gg 0$. 
Soit $n(x)$ le plus petit entier tel que 
$$
T^{n(x)} (x) \in \mathscr{S}_\eta^{ss}.
$$
Il y a une suite d'isogénies 
$$
H_x\drt H_{T(x)}\drt \cdots \drt H_{T^{n(x)}(x)}
$$
(il faudrait noter les extensions des scalaires puisque $K(T^{n(x)} (x))\subset\cdots \subset K(T(x))\subset K(x)$ et ces groupes ne sont pas définis sur les mêmes bases, mais afin d'alléger les notations on ne le fait pas).
On va montrer que 
$$
\ker \big ( H_x\drt H_{T^{n(x)}} \big ) \subset H_x [p^h].
$$
Supposons par l'absurde que ça ne soit pas le cas et soit $m\geq 0$ le plus grand entier tel que 
$$
I:=\ker \big ( H_{T^m(x)}\drt H_{T^{n(x)}(x)} \big ) \not\subset H_{T^m(x)}[p^h].
$$
D'après la proposition \ref{prop:desta clef deux suite retour} on a $0\leq m <n(x)-1$ et 
\begin{eqnarray*}
H_{T^m(x)}[p] &\not\subset  &I \\
I & \not \subset & H_{T^m(x)}[p^h] \\
I & \subset & H_{T^m(x)}[p^{2h}].
\end{eqnarray*} 
Soit $k$ le plus petit entier tel que 
$$
I\subset H_{T^m(x)}[p^k].
$$
On a donc $h<k\leq 2h$. Il y a alors une isogénie 
$$
H_{T^{n(x)}(x)}\simeq H_{T^m(x)}/I \xrig{\ \times p^k \ } H_{T^{m}(x)}.
$$
On vérifie alors que le noyau de cette isogénie est contenu dans $H_{T^{n(x)}(x)}[p^h]$ et ne contient pas $H_{T^{n(x)}(x)}[p]$. Cela implique que $T^{n(x)}(x)\in Y$ ce qui est impossible puisque $\mu_{\mathrm{max}} (T^{m}(x))<\a$.

Soit maintenant $\l'$ vérifiant $\frac{d}{h}<\l'<\l$. Considérons le domaine analytique compact
$$
K_{\l'}=\{ \mu_{\mathrm{max}} \leq \l'\} \subset \mathscr{S}_\eta.
$$
Puisque pour $x\in \mathscr{S}_\eta$, $HN (H_x)\leq HN (H_x[p])$ et $\HN(H_x)$ est un polygone à points de rupture de coordonnées entières, on a grâce à la définition de $\l$
$$
K_{\l'}\subset U.
$$
Considérons le domaine analytique compact 
$$
Z= \bigcup_{h\geq a_1\geq \cdots \geq a_h\geq 0 } Hecke_{\underline{a}} (\mathscr{S}_\eta^{ss}).
$$
On a montré que \og les points classiques de Tate\fg{} de $K_{\l'}$ sont contenus dans ceux de $Z$. On a donc, par densité des points classiques, $K_{\l'}\subset Z$. Ainsi, si $x\in \mathscr{S}_\eta$ vérifie $\mu_{\mathrm{max}} (H_x)<\l$ alors $H_x$ est isogène à un groupe $p$-divisible semi-stable.
\\
Soit maintenant $x\in U$ quelconque. D'après le point (2) du corollaire \ref{QDGFKIQDFDidgvoezrf249sfsf}, pour $n\gg 0$, $\mu_{\mathrm{max}} (H_{T^n (x)})<\l$. Cela permet de conclure.
\end{proof}

\begin{rema}
Au final, le point qui permet de conclure dans le théorème précédent est que, si $(\dim H,\htt \, H)=1$ et $H$ est semi-stable, alors $H[p]$ est stable. Cela est à rapprocher avec le fait que si $(d,r)=1$, alors tout fibré semi-stable de degré $d$ et de rang $r$ sur une surface de Riemann compacte est stable, la structure de l'espace de modules correspondant étant alors beaucoup plus simple (l'espace de modules des fibrés semi-stables est alors propre et lisse).
\end{rema}

Cela nous permet de conclure quant au théorème principal.

\begin{theo}\label{theo:vrai theo principal OK}
Soit $\M$ l'espace de Rapoport-Zink des déformations par quasi-isogénies d'un groupe $p$-divisible simple à isogénie près sur $\Fpb$ (\cite{RZ},\cite{RapoportViehmann}). Notons $\M_\eta$ sa fibre générique comme espace de Berkovich et  $$\pi_{dR}:\M_\eta\drt \F$$ l'application des périodes de Hodge-de-Rham, d'image le domaine de périodes $\F^a$. Soit $\M_\eta^{ss}$ le lieu où les points de $p$-torsion de la déformation universelle est un groupe plat fini semi-stable, un domaine analytique fermé. Alors, 
\begin{enumerate}
\item 
$
\pi_{dR} ( \M_\eta^{ss}) = \F^a.
$
En d'autres termes $Hecke.\M_\eta^{ss} = \M_\eta$,
\item le morphisme $\pi_{dR |\M_\eta^{ss}/p^\Z}$ est quasi-fini,
\item si $\M_{\eta}^{s}$ désigne le lieu stable, un ouvert de $\M_\eta$, alors 
$
\pi_{dR |\M_\eta^{s}/p^\Z}: \M_\eta^{s}/p^\Z\hookrightarrow \F^{a}.
$
\end{enumerate}
\end{theo}
\begin{proof}
L'assertion (1) est une conséquence du théorème \ref{theo:theo principal ou pas}. L'assertion (2) se déduit des propositions \ref{prop:destabilisation} et \ref{prop:astuce BT tronque 1}. 
Le point (3) se déduit du fait que $\pi_{dR | \M_\eta^s/p^\Z}$ est étale injectif au niveau des points géométriques.
\end{proof}

\section{Stratification de HN des Grassmaniennes et des variétés de Shimura}
\label{sec:stratification de HN des grassmaniennes}

\subsection{Cadre et définitions}

Soit $G=\GL_n$ sur $\Zp$. On note $\mu$ le cocaractère de $G$ défini par $\mu (z)=\mathrm{diag} (\underbrace{z,\cdots,z}_{d\text{-fois}},  1,\cdots, 1)$. On note $\mathrm{Perf}_{\Fp}$ la catégorie 
des $\Fp$-espaces perfectoïdes munie de la $v$-topologie. On utilise  la théorie des diamants de Scholze (\cite{ScholzeBerkeley}, \cite{ScholzeCohomologyDiamonds}). On va en effet définir certaines stratifications des Grassmaniennes $p$-adiques, dont les strates ne sont pas des espaces rigides usuels mais des diamants (on pourrait également les voir comme des espaces pseudo-adiques mais il est plus naturel d'utiliser la théorie des diamants de Scholze).

\begin{defi}
On note $\HT_{G,\mu}$ le $v$-champ sur $\mathrm{Perf}_{\Fp}$ qui à $S$ associe le groupoïde des quadruplets
$$
(S^\sharp,\F,\E,\a),
$$
où $S^\sharp$ est un débasculement de $S$, 
$\F$ est un $\Zp$-faisceau pro-étale localement constant libre de rang $n$ sur $S$, $\E$ un $\O_{S^\sharp}$-module localement libre de rang $n-d$ et $\a:\F\drt \E$ un morphisme $\Zp$-linéaire tel que $\F\otimes_{\underline{\Zp}} \O_{S^\sharp}\drt \E$ soit surjectif. 
\end{defi}

Dans la définition précédente on a pris quelques libertés quant aux notations. Plus précisément, si $\nu: S_{\text{pro{\'e}t}}\drt |S|$ est la projection du site pro-étale sur le site analytique, $\E$ est un faisceau sur $|S|$ et $\a:\F\drt \nu^*\E$. Afin d'alléger les notations, on utilise ce type de raccourcis dans la suite.

Si $\Gr_{n,n-d}$ désigne la grassmannienne des quotients localement libres de rang $n-d$ de $\O^n$ comme $\Qp$-espace adique, on a 
$$
\HT_{G,\mu} = \big [ \underline{\GL_n(\Zp)} \bc \Gr_{n,n-d}^\diamond\big ].
$$
Il s'agit donc d'un petit $v$-champ au sens de \cite{ScholzeCohomologyDiamonds}.

\begin{rema}
Plus généralement, supposons nous donné un groupe réductif $G$ sur $\Qp$ et une classe de conjugaison de cocaractère $\mu:\mathbb{G}_{m\overline{\Q}_p}\drt G_{\overline{\Q}_p}$. Soit $E$ le corps reflex de définition de $\{\mu\}$. Supposons fixé un sous-groupe compact ouvert $K\subset G(\Qp)$. On peut alors considérer le champ sur $\spa (E)^\diamond$ 
$$
\HT_{G,\mu,K} = \big [ \, \underline{K}\bc \mathrm{Gr}^{\leqslant \mu}_{B_{dR}} \,\big ],
$$
où $\mathrm{Gr}^{\leq \mu}_{B_{dR}}$ désigne la cellule de Schubert fermée associée à $\mu$ dans la $B_{dR}$-grassmannienne affine de Scholze (\cite{ScholzeBerkeley}). Nous nous restreignons au cas précédent du groupe linéaire et de $\mu$ minuscule pour une première approche concrète.
\end{rema}

Ce champ est muni de correspondances de Hecke. Pour chaque élément de $\GL_n (\Zp)\bc \GL_n (\Qp)/\GL_n (\Zp)$, une correspondance entre $\HT_{G,\mu}$ est lui-même formée des uplets
$(S^\sharp,\F,\E,\a,\F',\E',\a')$ et d'un isomorphisme $\F\unp \iso \F'\unp$,  tel que la position relative de $\F$ et $\F'$ soit donnée par notre double classe fixée.

\subsection{Stratification de Newton}

Il y a un morphisme de $v$-champs 
$$
\HT_{G,\mu}\ldrt \mathrm{Bun}_G,
$$
où $\mathrm{Bun}_G$ désigne le champ des fibrés de rang $n$ sur la courbe (\cite{GeometrizationReview}, \cite{geometrisation}). Celui-ci associe à $(S^\sharp,\F,\E,\a)$ le fibré $\G$ sur $X_S$, obtenu en modifiant le  fibré semi-stable de pente $0$, $\F\otimes_{\underline{\Zp}} \O_{X_S}$,
$$
0\ldrt \F\otimes_{\underline{\Zp}}\O_{X_S}\ldrt \G \ldrt i_* \ker (\a\otimes 1)(D)\ldrt 0
$$
où $i:S^\sharp \hookrightarrow X_S$ est défini par le débasculement $S^\sharp$ de $S$, $D$ est le diviseur de Cartier effectif associé à $i$ et
$\a\otimes 1: \F\otimes_{\underline{\Zp}} \O_{S^\sharp}\twoheadrightarrow \E$. Cette modification est obtenue par tiré en arrière de la modification
$$
0\ldrt \O_{X_S}(-D)\ldrt\O_{X_S}\ldrt i_* \O_{S^\sharp}\ldrt  0
$$
tordue par $\O(D)$ et tensorisée par application de $\F\otimes_{\underline{\Zp}}-$.
\\

 On peut dès lors tirer en arrière la stratification de Harder-Narasimhan de $\mathrm{Bun}_G$ et définir une {\it stratification de Newton de $\HT_{G,\mu}$}. Rappelons avant cela qu'à un élément $\nu=(\nu_1,\cdots,\nu_n)\in (\Q^n)^+$, on associe un polygone concave d'origine $(0,0)$ sur l'intervalle $[0,n]$ et de pente $\nu_i$ sur l'intervalle $[i-1,i]$.  Par définition, $\nu\leq \nu'$ si le polygone associé à $\nu$ est en dessous de celui associé à $\nu'$ et tous deux ont même points terminaux. Dans la suite, on va se restreindre aux $\nu$ tels que le polygone associé soit à points de rupture de coordonnées entières et de point terminal $(n,d)$.

\begin{defi}
Pour $\nu\in (\Q^n)^+$, 
on note $\HT_{G,\mu}^{\mathrm{Newt}^\diamond=\nu}$, resp. $\HT_{G,\mu}^{\mathrm{Newt}^\diamond\leqslant \nu}$, le sous-champ dont les $(C,C^+)$-points, $C$ algébriquement clos, sont formés des éléments de $\HT_{G,\mu} (C,C^+)$ tels que, si $\G$ est le fibré associé par modification et $\G\simeq \E(D,p^{-1}\ph)$, alors $\Newt (D,\ph)^{\diamond}=\nu$, resp. $\Newt (D,\ph)^{\diamond}\leq \nu$.
\end{defi}

Dit d'une autre façon, si $\HN (\G)$ est donné par $(\l_1,\cdots,\l_n)\in (\Q^n)^+$, 
$\Newt (D,\ph)^\diamond = (1-\l_n,\cdots,1-\l_1)$. 
Pour tout $\nu$, $\HT^{\Newt^\diamond\leqslant \nu}_{G,\mu}$ est un sous-champ ouvert partiellement propre de $\HT_{G,\mu}$, qui correspond à un ouvert $\GL_n (\Qp)$-invariant dans la grassmannienne $\mathrm{Gr}_{n,n-d}$. La stratification correspondante de cette grassmannienne est celle de Caraiani et Scholze (\cite{CaraianiScholze}).

\subsection{Stratification de Harder-Narasimhan}

\begin{defi}
On note $\HT_{G,\mu}^{\HN= \nu}$, resp. $\HT_{G,\mu}^{\HN\leqslant\nu}$,  le sous-champ dont les points sont formés des modules de Hodge-Tate rationnels (déf. \ref{defi:module de Hodge Tate}) de polygone de Harder-Narasimhan égal à $\nu$, resp. inférieur ou égal à $\nu$.
\end{defi}

On a le résultat de semi-continuité suivant.

\begin{prop}
Pour tout $\nu$, $\HT_{G,\mu}^{\HN \leqslant \nu}$ est un sous-champ ouvert partiellement propre dans $\HT_{G,\mu}$.
\end{prop}
\begin{proof}
Notons $\mathcal{P}:[0,n]\drt [0,d]$ le polygone concave associé à $\nu$. Pour $i\in\{1,\cdots,n-1\}$, soit 
$$
Z_i=\big  \{ x\in \mathrm{Gr}_{n,n-d}\ |\ \dim \big ( \mathrm{Im} ( K(x)^i\oplus (0)^{n-i}\drt K(x)^{n-d} ) \big )\leq n-\mathcal{P}(i)\big \},
$$
un fermé Zariski dans $\mathrm{Gr}_{n,n-d}$. Alors, 
$$
F= \bigcup_{i=1}^{n-1} \GL_n (\Qp).Z_i
$$
est fermé dans $|\mathrm{Gr}_{n,n-d}|$ par compacité de $\GL_n (\Qp)/\mathrm{Stab}_{\mathrm{GL}_n (\Qp)} (Z_i)$. On a alors 
\begin{equation*}
\HT_{G,\mu}^{\HN \leqslant \nu} = \big [ \, \underline{\GL_n (\Zp)}\,\bc\, \big ( \mathrm{Gr}_{n,n-d}^\diamond \smallsetminus F\big )\,\big ].\qedhere
\end{equation*}
\end{proof}

Les strates de Newton et de Harder-Narasimhan se comparent alors de la façon suivante. On note $\nu_{ss}$ le polygone qui est une droite de pente $d/n$.

\begin{prop}
Pour tout $\nu$, on a $$\HT_{G,\mu}^{\Newt^\diamond=\nu}\subset \HT_{G,\mu}^{\HN \leqslant \nu}.$$ En particulier, la strate ouverte \og basique\fg{} $\HT_{G,\mu}^{\Newt^\diamond=\nu_{ss}}$ est contenue dans la strate ouverte \og semi-stable\fg{} 
$\HT_{G,\mu}^{\HN=\nu_{ss}}$.
\end{prop}

\subsection{Lieu de type HN entier}

Les deux stratifications définies précédemment sont Hecke invariantes. On va maintenant définir des \og domaines fondamentaux\fg{} pour l'action des correspondances de Hecke dans chaque strate de Harder-Narasimhan. 

À chaque $x\in \Gr_{n,n-d}$, est associé un morphisme $\a_x:\Zp^n\drt K(x)^{n-d}$. On note $\omega_x$ le sous-$K(x)^0$-module engendré par l'image de ce morphisme. Cela définit un module de Hodge-Tate entier (déf. \ref{defi:module de Hodge Tate})
$$
X_x = (\Zp^n,\omega_x,\a_x) \in \M^{\HT}_{\Zp}.
$$
On munit $K(x)$ de la valuation de rang un associée à la généralisation maximale de $x$ dans notre espace adique. Comme d'habitude, on normalise cette valuation de telle manière que $v(p)=1$. Grâce à cela on définit des fonctions continues pour $n\geq 1$,
$$
x\mapsto \HN (X_x[p^n]),
$$
de $|\mathrm{Gr}_{n,n-d}|$ à valeurs dans les polygones (ces fonctions continues se factorisent en fait par l'espace topologique de Berkovich associé).

\begin{defi}
Pour un polygone $\nu$, on note $\HT_{G,\mu}^{\HN=\nu,ss}$ le sous-champ associé aux $x\in \mathrm{Gr}^{\HN=\nu}_{n,n-d}$ tels que $X_x$ soit de type HN, i.e. $\HN (X_x[p])=\nu$.
\end{defi}

Contrairement à la strate $\HN^{\HN=\nu}$ qui n'est pas quasi-compacte en général, ni associée à un ouvert admissible de la grassmannienne rigide analytique, on a le résultat suivant.

\begin{lemm}
L'espace topologique $|\HT_{G,\mu}^{\HN=\nu,ss}|$ est un fermé quasi-compact de la forme $(\GL_n (\Zp) \bc |\overline{U}|)\cap |\HT_{G,\mu}^{\HN\geq \nu}|$, avec $U$ un ouvert quasi-compact de $\mathrm{Gr}_{n,n-d}$.
\end{lemm}
\begin{proof}
C'est une conséquence du fait que, $\HN(X_x\unp)=\nu$ et $\HN (X_x[p])=\nu$, est équivalent à ce que, $\HN(X_x\unp)\geq \nu$ et $\HN (X_x[p])=\nu$, cf. théo. \ref{SFKEGI39459sguzSF20refq}.
\end{proof}

D'après le théorème \ref{theo:essentiel}, on a un recouvrement 
\begin{equation}\label{eq:recouvrement}
|\HT_{G,\mu}^{\HN=\nu}| = \bigcup_{\underline{a}\in (\Z^n)^+/\Z} Hecke_{\underline{a}} \cdot |\HT_{G,\mu}^{\HN=\nu,ss}|.
\end{equation}

\subsection{Induction parabolique}

On dispose de l'énoncé suivant de \og mise en famille\fg{} des filtrations de Harder-Narasimhan des modules de Hodge-Tate.

\begin{prop}\label{prop:induction parabolique I}
Soit $\nu\in (\Q^n)^+$ associé à un polygone de pentes $\l_1> \cdots > \l_r$ sur  les intervalles $[0,h_1],\dots, [h_{r-1},h_r]$, $h_r=n$. Le champ $\HT_{G,\mu}^{\HN=\nu}$ s'identifie alors au champ qui à $S\in \mathrm{Perf}_{\Fp}$ associe le groupoïde des quadruplets $(S^\sharp,\F_\bullet, \E_\bullet,\a)$ où 
\begin{itemize}
\item $\F_\bullet$ est un $\Zp$-faisceau pro-étale filtré sur $S$ 
$$
0=\F_0\subset \F_1\subset \cdots\subset  \F_r 
$$
avec
\begin{itemize}
\item pour tout $i$, $\F_i$ localement constant libre de rang $h_i$ 
\item la filtration précédente est pro-étale localement scindée
\end{itemize}
\item $\E_\bullet$ est un fibré vectoriel filtré sur $S^\sharp$
$$
0=\E_0\subset \E_1\subset \cdots \subset \E_r
$$
avec
\begin{itemize}
\item pour tout $i$, $\E_i$ localement libre de rang $h_1(1-\l_1)+\cdots +h_i (1-\l_i)$
\item la filtration précédente est localement scindée sur $S^\sharp$
\end{itemize}
\item $\a: \F_\bullet \drt \E_\bullet$ est tel que pour tout $i$, 
$$(\a_{|\F_i}\otimes 1): \F_i\otimes_{\underline{\Zp}}\O_{S^\sharp} \ldrt \E_i
$$
est surjectif
\item pour tout $i\geq 1$, le gradué $\F_i/\F_{i-1}\drt \E_{i}/\E_{i-1}$ est semi-stable de pente $\l_i$, fibre à fibre sur $S$.
\end{itemize}
\end{prop}
\begin{proof}
Dans cette preuve, la Grassmanienne est considérée en tant que schéma.
Soit $V=\Qp^{n}$ muni de la filtration $(V_i)_{0\leq i\leq r}$ telle que $V_i=\Qp^{h_i}\oplus (0)$. Notons $P$ le sous-groupe parabolique de $\GL_n$ associé. Soit $Y\subset \mathrm{Gr}_{n,n-d}$, le sous-schéma réduit localement fermé formé des quotients localement libres de rang $n-d$ de $\O^n$ tels que pour tout point  $x$, l'image de $k(x)^{h_i}\oplus (0)$ soit de rang $h_1(1-\l_1)+\cdots +h_i (1-\l_i)$. Notons $\O^{n}\twoheadrightarrow \E$ le quotient universel sur $\mathrm{Gr}$. Alors,  pour tout $i$, l'image de $\O_Y^{h_i}\oplus (0)\drt \E_{|Y}$ est localement libre. Dès lors, 
$$
\underline{\GL_n (\Zp)/ P(\Zp)}\times Y^{ad,\diamond} \hookrightarrow \mathrm{Gr}_{n,n-d}^{ad,\diamond}
$$
et 
$$
\HT_{G,\mu}^{\HN=\nu} = \big [\, \underline{\GL_n (\Zp)}\, \bc\, Z\,\big ],
$$
avec $Z\subset \underline{\GL_n (\Zp)/ P(\Zp)}\times Y^{ad,\diamond}$ un sous-diamant fermé.
\end{proof}

Soit maintenant 
$$
M=\GL_{h_1}\times \GL_{h_2-h_1}\times \cdots\times \GL_{h_r-h_{r-1}}
$$
comme sous-groupe de Levi de $\GL_n$. On note $\mu$ le cocharactère de $M$ dont la composante sur le facteur $\GL_{h_i-h_{i-1}}$ est $\mathrm{diag} (\underbrace{z,\cdots,z}_{(h_i-h_{i-1})(1-\l_i)} ,1,\cdots, 1)$. De la proposition précédente on déduit le résultat suivant.

\begin{prop}\label{prop:calcul extensions}
Il y a un morphisme 
$$
\HT_{G,\mu}^{\HN=\nu}\ldrt \HN_{M,\mu}^{\HN=\nu_{ss}}
$$
qui est une extension successive de champs de Picard pro-étales de la forme 
$$
\big [ \, \F \drt \E \, \big ]
$$  
où $\F$ est un $\Zp$-faisceau pro-étale localement constant de rang fini, $\E$ un $\O^\sharp$-module localement libre de rang fini et $\F\drt \E$ est $\Zp$-linéaire.
\end{prop}
\begin{proof}
Le morphisme de champs est donné par la proposition \ref{prop:induction parabolique I} en \og passant aux gradués\fg{}. 
Soit $S\in \mathrm{Perf}_{\Fp}$, $S^\sharp$ un débasculement de $S$, $(\F_1,\E_1,\a_1)$ et $(\F_2,\E_2,\a_2)$ deux $\Zp$-modules de Hodge-Tate entiers sur $S^\sharp$. On cherche à calculer le champ de Picard des extensions entre $(\F_2,\E_2,\a_2)$ et $(\F_1,\E_1,\a_1)$. 
Soit donc une extension
$$
\xymatrix@C=8mm@R=8mm{ 
0\ar[r] & \F_2\ar[d]_{\a_2}\ar[r] & \F\ar[r]\ar[d] & \F_1\ar[r]\ar[d]^{\a_1} & 0 \\
0\ar[r] & \E_2\ar[r] & \E\ar[r] & \E_1\ar[r] & 0.
}
$$
Pro-étale localement sur $S$, les suites exactes du haut et du bas sont scindées. Après scindage, le morphisme $\F\drt \E$ est donné par un morphisme $\F_1\drt \E_2$. De tels scindages forment un torseur pro-étale sous $\Hom_{\Zp} (\F_1,\F_2)\oplus \Hom_{\O_{S^\sharp}} ( \E_1,\E_2)$. On en déduit que notre champ de Picard est associé au complexe
\begin{equation*}
\Hom_{\Zp} (\F_1,\F_2)\oplus \Hom_{\O_{S^\sharp}} (\E_1,\E_2)\xrig{\ \a_{2*}\oplus \a_{1}^*} \Hom_{\Zp} (\F_1,\E_2). 
\end{equation*}
Puisque $\a_1\otimes 1$ est surjectif, $\a_1^*$ est injectif et le complexe précédent est quasi-isomorphe à
\begin{equation}\label{eq:eq10}
\Hom_{\Zp} (\F_1,\F_2) \ldrt \Hom_{\Zp} ( \F_1,\E_2) / \a_1^*\Hom_{\O_{S^\sharp}} (\E_1,\E_2).
\end{equation}
On conclut avec l'égalité 
\begin{equation*}
\Hom_{\Zp} ( \F_1,\E_2) / \a_1^*\Hom_{\O_{S^\sharp}} (\E_1,\E_2) = \Hom_{\O_{S^\sharp}} ( \ker (\a_1\otimes 1),\E_2). \qedhere
\end{equation*}
\end{proof}

On utilise les notions de dimension de la section 21 de \cite{ScholzeCohomologyDiamonds}.

\begin{prop}\label{prop:calcul dimension strates HN}
On a l'égalité 
$$
\dim |\HT_{G,\mu}^{\HN=\nu}| = <\mu-\nu,2\rho>
$$
qui coïncide avec $\mathrm{dim.\ trg} \ \HT_{G,\mu}^{\HN=\nu}$.
\end{prop}
\begin{proof}
On sait que $\HN_{M,\mu}^{\HN=\nu_{ss}}$ est un ouvert de $\HN_{M,\mu}$ dont on connait la dimension. Il suffit alors de 
reprendre le calcul d'extensions de la preuve de la proposition \ref{prop:calcul extensions}. 
La dimension du champ de Picard associé au complexe de l'équation (\ref{eq:eq10}) est 
$$
\big (\mathrm{rg}_{\Zp} (\F_1 ) - \mathrm{rg}_{\O_{S^\sharp}} (\E_1 ) \big ). \mathrm{rg}_{\O_{S^\sharp}} (\E_2).
$$
A partir de là le calcul se fait sans difficultés.
\end{proof}

\begin{exem}
Prenons $d=n-1$. Dès lors, les stratifications de Newton et de Harder-Narasimhan coïncident. La strate ouverte est $\Omega^{n-1}\subset \P^{n-1}$. La strate où le noyau de $\underline{\Qp^n}\drt \O(1)$ est de dimension $n-i$ est isomorphe à $\underline{\Gr_{n,i}(\Qp)}\times \Omega^{i,\diamond}$.
\end{exem}

\subsection{Stratification des variétés de Shimura}
\label{sec:stratifications varietes de Shimura}

Soit $\mathscr{S}$ le complété $p$-adique du modèle entier en niveau hyperspécial en $p$ d'une variété de Shimura de type PEL (\cite{Ko1}). On suppose le niveau hors $p$ fixé, suffisamment petit. On suppose, de plus, que le groupe en $p$ associé est isomorphe à $\GL_{n\Qp}\times \mathbb{G}_{m\Qp}$, le facteur en $p$ correspondant au facteur de similitude. Via cette identification, on fait l'hypothèse que le cocaractère de Hodge est donné par 
$
\mu (z)=\mathrm{diag} (\underbrace{z,\cdots,z}_{d\text{-fois}},  1,\cdots, 1)\times (z)$. 
 Notons $S$ la fibre générique de $\mathscr{S}$ comme $\Qp$-espace adique. Il y a un morphisme de périodes de Hodge-Tate (\cite{LivreIso}, \cite{ScholzeTorsion})
 $$
 \pi_{\HT}: S^\diamond\ldrt \HT_{G,\mu}.
 $$
Pour tout $\nu$, 
$$
\pi_{\HT}^{-1}\big (\HT_{G,\mu}^{\Newt^\diamond\leqslant  \nu} \big ) 
$$
est le diamant du  tube au dessus de l'union finie de strates de Newton de $\mathscr{S}_{\Fp}$ de polygone (concave) plus petit que $\nu$. Par tiré en arrière on obtient deux stratifications Hecke invariantes
$$
S^{\diamond, \Newt^\diamond=\nu} , \ S^{\diamond,\HN=\nu},
$$
indéxées par de tels $\nu$ avec des inclusions d'ouverts
$$
S^{\diamond,\Newt^\diamond \leqslant \nu} \subset S^{\diamond, \HN \leqslant \nu}.
$$
On note $S_\infty$ la variétés de Shimura en niveau infini en $p$ en tant que $\Qp$-espace perfectoïde. Soit $\nu\neq \nu_{ss}$, le polygone qui est une droite de pente $\frac{d}{n}$. On note $P_\nu$ le sous-groupe parabolique de $\GL_{n\Qp}$ associé à $\nu$.
 On a alors le résultat suivant qui découle de la proposition \ref{prop:induction parabolique I}.

\begin{prop}
Il existe un diamant localement spatial $T_\nu$ sur $\spa (\Qp)$, muni d'une action de $P_\nu (\Qp)$ et de correspondances de Hecke hors $p$, tel que 
$$
S_\infty^{\HN=\nu} = T_\nu \underset{\underline{P_\nu (\Qp)}}{\times} \underline{G(\Qp)}.
$$
\end{prop}

En d'autres termes, les  strates non semi-stables sont paraboliquement induites.

\subsection{Perspectives}\label{sec:perspectives}
\subsubsection{Extension à d'autres groupes et au cas non minuscule}

Soit $G$ un groupe réductif sur $\Qp$ et $\mu:\mathbb{G}_{m\Qpb}\drt G_{\Qpb}$ non forcément minuscule. 
Les résultats de \cite{CornutMacarena} devraient permettre la construction d'une stratification de Harder-Narasimhan $G(\Qp)$-invariante de 
$$
\mathrm{Gr}_{\mathrm{B}_{dR}}^{\leqslant \mu}
$$
la cellule de Schubert fermée associée à $\mu$ sur $\spa (E)^\diamond$. D'après \cite{Gtorseurs} et \cite{CaraianiScholze} on dispose d'une stratification de Newton $G(\Qp)$-invariante de cette cellule indexée par l'ensemble de Kottwitz $B(G,\mu^{-1})$. 
On peut alors espérer que la stratification de Harder-Narasimhan soit indexée par le même type d'ensemble avec des inclusions
$$
(\mathrm{Gr}_{\mathrm{B}_{dR}}^{\leqslant \mu})^{\Newt^\diamond \leq \nu}\subset (\mathrm{Gr}_{\mathrm{B}_{dR}}^{\leqslant \mu})^{\HN\leq \nu}
$$
pour $\nu$ dans une chambre de Weyl positive. La conjecture suivante devrait alors 
résulter de techniques identiques à celles utilisées dans \cite{ChenFarguesXu}.

\begin{conj}
\

\begin{enumerate}
\item 
Sont équivalents:
\begin{enumerate}
\item Les stratifications de Harder-Narasimhan et de Newton de $\mathrm{Gr}_{\mathrm{B}_{dR}}^{\leqslant \mu}$ coïncident.
\item L'inclusion $(\mathrm{Gr}_{\mathrm{B}_{dR}}^{\leqslant \mu})^{\Newt^\diamond=\nu_{ss}} \subset (\mathrm{Gr}_{\mathrm{B}_{dR}}^{\leqslant \mu})^{\HN=\nu_{ss}}$ de l'ouvert basique dans l'ouvert semi-stable est une égalité.
\item L'ensemble de Kottwitz $B(G,\mu^{-1})$ est pleinement Hodge-Newton décomposable (\cite{GoertzHeNie}, \cite{ChenFarguesXu}). 
\end{enumerate}
\item 
Les strates fermées \og $\mu$-ordinaire\fg{} de Newton et de Harder-Narasimhan coïncident toujours. 
\end{enumerate}
\end{conj}

Citons également la conjecture suivante extension des propositions \ref{prop:induction parabolique I} et \ref{prop:calcul dimension strates HN}.

\begin{conj}
\

\begin{enumerate}
\item 
Les strates de Harder-Narasimhan non semi-stables sont paraboliquement induites. 
\item La dimension de la strate associé au vecteur de Harder-Narasimhan $\nu$ est donnée par 
le produit scalaire $<\mu-\nu,2\rho>$.
\end{enumerate}
\end{conj}

\subsubsection{Applications arithmétiques}

Il s'agirait ici d'utiliser les techniques de \cite{CaraianiScholze} couplées à la filtration de Harder-Narasimhan précédente. Plus précisément, on peut considérer le complexe $G(\Qp)\times (\text{Hecke hors }p)$-équivariant 
$$
R\pi_{\HT*} \overline{\mathbb{F}}_\ell.
$$
Cariani et Scholze montrent qu'il satisfait une certaine condition de perversité relativement à la stratification de Newton de la variété de drapeaux (\cite{CaraianiScholze} prop. 6.1.3). Ils calculent de plus les fibres géométriques de ce complexe en termes de cohomologie de variétés d'Igusa (\cite{CaraianiScholze} théo. 4.4.4). Il serait intéressant de regarder la restriction de ce complexe aux strates de Harder-Narasimhan. Par exemple, la partie supercuspidale en $p$ de ce \og complexe pervers\fg{},
$$
(R\pi_{\HT*} \overline{\mathbb{F}}_\ell)_{cusp},
$$
qui devrait encore être pervers, 
devrait être concentrée sur la strate de Harder-Narasimhan semi-stable (propriété d'induction parabolique des strates non semi-stables). On peut alors se poser la question de savoir quelles contraintes cela impose lorsque l'on couple cette propriété de localisation avec la propriété de \og perversité\fg{} et le calcul de la cohomologie des fibres géométriques, qui sont constantes le long de chaque strate de Newton.

Supposons par exemple que l'on sache montrer que:
\begin{enumerate}
\item Pour tout $i$ et $\nu\neq \nu_{ss}$, $(R^i\pi_{\HT *} \overline{\mathbb{F}}_\ell)_{\mathcal{F}l^{\Newt^\diamond =\nu}}$ est un système local étale (Caraiani et Scholze montrent seulement que les fibres géométriques sont constantes).
\item Pour tout $\nu \neq \nu_{ss}$ et toute composante connexe $\mathcal{C}$ de $\mathcal{F}l^{\Newt^\diamond=\nu}$, $\mathcal{C}\not \subset \mathcal{F}l^{\HN=\nu_{ss}}$. 
\end{enumerate}

On devrait alors pouvoir montrer que $(R\pi_{\HT*} \overline{\mathbb{F}}_\ell)_{cusp}$ est concentré sur le lieu basique et donc la partie supercuspidale en $p$ de la cohomologie de la variété de Shimura, coïncide avec la partie supercuspidale de la cohomologie du lieu basique (\cite{Laurent1} et \cite{ShinRZ} pour des cas particuliers au niveau des sommes alternées dans des groupes de Grothendieck).

\bibliographystyle{plain}
\bibliography{../../../biblio.bib}
\end{document}